\DeclareSymbolFont{rsfs}{U}{rsfs}{m}{n}

\documentclass[10pt,a4paper,oneside,leqno,noamsfonts]{amsart}
           \makeatletter
           \newcommand{\mylabel}[2]{#2\def\@currentlabel{#2}\label{#1}}
           \renewcommand\@biblabel[1]{#1.}
           \makeatother
\usepackage[english]{babel}
\usepackage[dvipsnames]{xcolor}

\usepackage[noadjust]{cite}

\usepackage{graphicx,pifont}
\usepackage[a4paper,top=2.8cm,bottom=2.8cm,left=2.7cm,
           right=2.7cm,bindingoffset=5mm,marginparwidth=60pt]{geometry}
\usepackage[utf8]{inputenc}
\usepackage{braket,caption,comment,mathtools,stmaryrd}
\usepackage[usestackEOL]{stackengine}
\usepackage{multirow,booktabs,microtype,relsize}
\usepackage[colorlinks,bookmarks]{hyperref} % ,backref=page
      \hypersetup{colorlinks,%
            citecolor=britishracinggreen,%
            filecolor=black,%
            linkcolor=cobalt,%
            urlcolor=black}
      \numberwithin{equation}{section}
\usepackage{mathrsfs}
\usepackage[colorinlistoftodos]{todonotes}%textwidth=35mm

\definecolor{antiquewhite}{rgb}{0.98, 0.92, 0.84}
\definecolor{buff}{rgb}{0.94, 0.86, 0.51}
\definecolor{palecopper}{rgb}{0.85, 0.54, 0.4}
\definecolor{fluorescentyellow}{rgb}{0.8, 1.0, 0.0}

\definecolor{britishracinggreen}{rgb}{0.0, 0.26, 0.15}
\definecolor{cobalt}{rgb}{0.0, 0.28, 0.67}
\DeclareSymbolFont{usualmathcal}{OMS}{cmsy}{m}{n}
\DeclareSymbolFontAlphabet{\mathcal}{usualmathcal}

\newcommand{\TT}{\mathbf{T}}

\newcommand{\BC}{{\mathbb{C}}}

\newcommand{\BE}{{\mathbb{E}}}

\newcommand{\BL}{{\mathbb{L}}}

\newcommand{\BP}{{\mathbb{P}}}
\newcommand{\BQ}{{\mathbb{Q}}}
\newcommand{\BR}{{\mathbb{R}}}

\newcommand{\BT}{{\mathbb{T}}}

\newcommand{\BZ}{{\mathbb{Z}}}

\newcommand{\CA}{{\mathcal A}}

\newcommand{\CE}{{\mathcal E}}
\newcommand{\CF}{{\mathcal F}}

\newcommand{\CI}{{\mathcal I}}

\newcommand{\CL}{{\mathcal L}}
\newcommand{\CM}{{\mathcal M}}

\newcommand{\CZ}{{\mathcal Z}}

\newcommand{\simto}{\,\widetilde{\to}\,}

\newcommand{\pt}{{\mathsf{pt}}}

\newcommand{\fix}{\mathsf{fix}}
\newcommand{\mov}{\mathsf{mov}}

\DeclareMathOperator{\Hilb}{Hilb}

\DeclareMathOperator{\Quot}{Quot}

\DeclareMathOperator{\coker}{coker}

\DeclareMathOperator{\Sym}{Sym}

\DeclareMathOperator{\coh}{coh}

\DeclareMathOperator{\vir}{\mathrm{vir}}

\DeclareMathOperator{\Exp}{Exp}
\DeclareMathOperator{\Var}{Var}

\DeclareMathOperator{\GL}{GL}
\DeclareMathOperator{\U}{U}

\DeclareMathOperator{\rk}{rk}

\DeclareMathOperator{\tr}{tr}

\newcommand{\derived}{\mathbf{D}}

\newcommand*{\defeq}{\mathrel{\vcenter{\baselineskip0.5ex \lineskiplimit0pt
                     \hbox{\scriptsize.}\hbox{\scriptsize.}}}%
                     =}

\usepackage{bbm}

\newcommand{\nc}{\ensuremath{\rm nc}}

\newcommand{\Rep}{\ensuremath{\mathsf{Rep}}}
\newcommand\LQ{\ensuremath{\mathsf{L_4}}}
\newcommand\tLQ{\ensuremath{\widetilde{\mathsf{L}}_4}}
\usepackage{subcaption}
\usepackage[outline]{contour}
\usetikzlibrary{cd,arrows,decorations.markings,positioning,shapes,calc,patterns,tqft,intersections,shapes.misc}
\tikzset{>=Latex}
\tikzset{GaugeNode/.style={circle,draw,inner sep=0pt,minimum size=10mm}}
\tikzset{FrameNode/.style={rectangle,draw,inner sep=0pt,minimum size=9mm}}
\tikzset{token/.style={circle,double,draw=black!70,fill=black!50,inner sep=0pt,minimum size=3mm}}
\newcommand{\into}{\hookrightarrow}
\newcommand{\onto}{\twoheadrightarrow}

\DeclareFontFamily{OT1}{rsfs}{}
\DeclareFontShape{OT1}{rsfs}{n}{it}{<-> rsfs10}{}
\DeclareMathAlphabet{\curly}{OT1}{rsfs}{n}{it}

\newcommand\Hom{\operatorname{Hom}}
\newcommand\End{\operatorname{End}}

\newcommand\id{\operatorname{id}}

\newcommand{\OO}{\mathscr O}



\usepackage[all]{xy}
\usepackage{tikz}
\usepackage{tikz-cd}
\usepackage{adjustbox}
\usepackage{rotating}
\usepackage{comment}

\usetikzlibrary{matrix,shapes,intersections,arrows,decorations.pathmorphing}
\tikzset{commutative diagrams/arrow style=math font}
\tikzset{commutative diagrams/.cd,
mysymbol/.style={start anchor=center,end anchor=center,draw=none}}

\tikzset{
shift up/.style={
to path={([yshift=#1]\tikztostart.east) -- ([yshift=#1]\tikztotarget.west) \tikztonodes}
}
}

\theoremstyle{definition}

\newtheorem*{lemma*}{Lemma}
\newtheorem*{theorem*}{Theorem}
\newtheorem*{example*}{Example}
\newtheorem*{fact*}{Fact}
\newtheorem*{notation*}{Notation}
\newtheorem*{definition*}{Definition}
\newtheorem*{prop*}{Proposition}
\newtheorem*{remark*}{Remark}
\newtheorem*{corollary*}{Corollary}

\newtheorem*{conventions*}{Conventions}

\newtheorem{definition}{Definition}[section]

\newtheorem{example}[definition]{Example}

\newtheorem{remark}[definition]{Remark}

\newtheoremstyle{thm} % <name> % (ambienti con dimostrazione)
        {3mm}% <Space above>
        {3mm}% <Space below>
        {\slshape}% <Body font> % 
        {0mm}% <Indent amount>
        {\bfseries}% <Theorem head font>
        {.}% <Punctuation after theorem head>
        {1mm}% <Space after theorem head>
        {}% <Theorem head spec (can be left empty, meaning 'normal')> 
\theoremstyle{thm}
\newtheorem{theorem}[definition]{Theorem}
\newtheorem{corollary}[definition]{Corollary}
\newtheorem{lemma}[definition]{Lemma}
\newtheorem{prop}[definition]{Proposition}

\newtheoremstyle{ex} % <name> % (ambienti con dimostrazione)
        {3mm}% <Space above>
        {3mm}% <Space below>
        {}% <Body font> % \slshape
        {0mm}% <Indent amount>
        {\scshape}% <Theorem head font>
        {.}% <Punctuation after theorem head>
        {1mm}% <Space after theorem head>
        {}% <Theorem head spec (can be left empty, meaning 'normal')> 
\theoremstyle{ex}

\newtheoremstyle{sol} % <name> % (ambienti con dimostrazione)
        {3mm}% <Space above>
        {3mm}% <Space below>
        {}% <Body font> % 
        {0mm}% <Indent amount>
        {\scshape}% <Theorem head font>
        {.}% <Punctuation after theorem head>
        {1mm}% <Space after theorem head>
        {}% <Theorem head spec (can be left empty, meaning 'normal')> 
\theoremstyle{sol}

\usepackage{amsmath,float}
\usetikzlibrary{decorations.markings}
\DeclareMathOperator{\Spin}{Spin}

\usepackage{amssymb}
\usepackage{enumitem}
\usepackage{comment}
\newcommand{\sgn}{{\mathrm{sgn}}}
\usepackage{bm}

   \DeclareMathOperator{\oO}{\mathcal{O}}
   
   \DeclareMathOperator{\tf}{\mathfrak{t}}

\allowdisplaybreaks
\usepackage{youngtab} % youngtab want label with one entry, so I define some extra commands

\usepackage{ytableau} %better for the labelling
\title[Tetrahedron instantons in Donaldson-Thomas theory]{Tetrahedron instantons in Donaldson-Thomas theory}

\author{Nadir Fasola, Sergej Monavari}
\address{School of Mathematics and Statistics, University of Sheffield, Hounsfield Road, Sheffield, S3 7RH, United Kingdom}
\email{n.fasola@sheffield.ac.uk}
\address{Ecole Polytechnique Fédérale de Lausanne (EPFL),  CH-1015 Lausanne, Switzerland}
\email{sergej.monavari@epfl.ch}

\begin{document}
\maketitle
\begin{abstract}
Inspired by the work of Pomoni-Yan-Zhang in String Theory, we introduce the moduli space of tetrahedron instantons as a Quot scheme on a singular threefold and describe it as a moduli space of quiver representations. We construct a virtual fundamental class and virtual structure sheaf à la Oh-Thomas, by which we define $K$-theoretic invariants. We show that the partition function of such invariants reproduces the one studied by Pomoni-Yan-Zhang, and explicitly determine it, as a product of shifted partition functions of rank one  Donaldson-Thomas invariants of the  three-dimensional affine space. Our geometric construction answers a series of questions of Pomoni-Yan-Zhang on the geometry of the moduli space of tetrahedron instantons and the behaviour of its partition function, and provides a new application of the recent work of Oh-Thomas.
\end{abstract}
\section{Introduction}
 \subsection{Tetrahedron instantons}\label{sec:intro tetrahedra}
 Pomoni-Yan-Zhang \cite{PYZ_tetrahedron} recently introduced the \emph{tetrahedron instantons}, which can be realized in type IIB  string theory as D1-branes probing  a configuration of intersecting D7-branes in flat spacetime with constant $B$-field turned on.
 %in a flat spacetime with a  constant $B$-field turned on.  
 % In other words, they describe instantons on $\BC^3$ in the presence of real codimension-two supersymmetric defects. 
 This is equivalently interpreted as describing instanton dynamics on $\BC^3$ in the presence of real-codimension-two supersymmetric defects.
 
 Their construction of the moduli space of tetrahedron instantons is realised via gauge-theoretic methods, so that  supersymmetric localisation techniques apply to define and  study  the associated instanton partition function non-perturbatively. In this article, we propose a rigorous mathematical perspective on the problem. We realise the moduli space of tetrahedron instantons as a Quot scheme and compute explicitly the generating series of the resulting virtual  invariants, thus shedding light on the geometry of the constructions in \cite{PYZ_tetrahedron} from the point of view of Donaldson-Thomas theory.
 \subsubsection{Mathematical formulation}
The idea of mathematically studying instanton partition functions --- defined in \emph{topological string theory} --- via moduli spaces of sheaves (in our case, Quot schemes) is not new in the literature, and follows the circle of ideas for which \emph{systems of D-branes} should be identified with (complexes of) sheaves with prescribed Chern character and suitable stability condition. In this setting, the instanton partition function is usually realised  as a generating series of \emph{virtual} invariants on the corresponding moduli space and exactly computed via (virtual) localisation techniques. See \cite{Okounk_Lectures_K_theory, Arb_K-theo_surface, FMR_higher_rank,  CKM_K_theoretic, Mon_canonical_vertex, CKM_crepant, KR_magnificient, Mon_PhD} for some recent examples 
where this circle of ideas has been  applied, providing an interpretation in Donaldson-Thomas theory of  some results initially introduced in topological string theory (for instance \cite{IKV_topological_vertex,  Nek_Z_theory, AK_quiver_matrix_model, BBPT_elliptic_DT, Nek_magnificient_4, NP_colors}). A novel feature of our proposal is that the relevant moduli space we study  --- \emph{geometrising} the work of \cite{PYZ_tetrahedron} --- is the Quot scheme of a \emph{singular} variety, parametrising quotients of a \emph{torsion} sheaf, in contrast to the classical case of torsion-free sheaves on smooth varieties.
\subsubsection{The moduli space}
  For $i=1,\dots, 4$  denote the $i$-th coordinate hyperplane in $\BC^4$ by $\BC^3_{i}:=Z(x_i)$ and  set $\Delta=\bigcup_{i=1}^4 \BC^3_i \subset \BC^4$ to be the union of the four coordinate hyperplanes. For a tuple $\overline{r}=(r_1, r_2, r_3, r_4)$ of non-negative integers, define the torsion sheaf on $\Delta$
 \[\CE_{\overline{r}}=\bigoplus\iota_{i,*}\oO^{\oplus{r_i}}_{\BC^3_i},\]
 where $\iota_i: \BC^3_i\hookrightarrow \Delta$ are the inclusions of the irreducible components.
Our main moduli space of interest is Grothendieck's \emph{Quot scheme}
 \[\CM_{\overline{r}, n}:=\Quot_\Delta(\CE_{\overline{r}},n),\]
 which parametrizes isomorphism classes of quotients $[\CE_{\overline{r}}\onto Q]$, where we identify two such  quotients if they have the same kernel. We refer to  $\CM_{\overline{r}, n} $ as the \emph{moduli space of tetrahedron instantons}. The term \emph{tetrahedron} comes from the fact that 
 the singular variety $\Delta$ appears as the central fibre of a toric degeneration, whose associated polyhedral complex is a tetrahedron\footnote{We thank Francesca Carocci for explaining this to us from the point of view of Tropical Geometry (cf.~also \cite[Fig.~1]{PYZ_tetrahedron}).}.

 The geometry of $\CM_{\overline{r}, n}$ is subtle and hard to explicitly describe scheme-theoretically. In fact, already  in the easiest example  $r=(1,0,0,0)$, the moduli space of tetrahedron instantons is isomorphic to the \emph{Hilbert scheme of points} $\Hilb^n(\BC^3)$, whose actual dimension and number of irreducible components are not currently known.\\

Our first result is that the moduli space of tetrahedron instantons $ \CM_{\overline{r}, n}$ can be realized as the zero-locus of an isotropic section of a \emph{special orthogonal bundle} (cf.~Sec.~\ref{sec: virtual classes}) inside a smooth ambient space.
 \begin{theorem}[Theorem \ref{thm: isotropic construction}]\label{thm: zero locus intro}
    Let $\overline{r}=(r_1, r_2, r_3, r_4)$ and $n$ be  non-negative integers. Then
    \[
\begin{tikzcd}
& \CL\arrow[d]\\
\CM_{\overline{r}, n}\cong Z(s)\arrow[r, hook, "\iota"] &\CM^{\nc}_{\overline{r}, n},\arrow[u, bend right, swap, "s"]
\end{tikzcd}
\]
where $\CM^{\nc}_{\overline{r}, n}$ is a smooth variety, $\CL$ is a special orthogonal bundle, and  $s\in H^0(\CM_{\overline r,n}^{\nc},\CL)$ is an isotropic section.
\end{theorem}
We denote the smooth ambient space  $ \CM^{\nc}_{\overline{r}, n}$ by  \emph{non-commutative Quot scheme}, in accordance with  \cite{Ric_noncomm}, who computed its motivic classes in $K_0(\Var_\mathbf{k})$ in any dimension and over any algebraically closed field $\mathbf{k}$. Similar structural results --- exploiting  non-commutative Quot schemes --- appeared in \cite{BR_higher_rank, KR_magnificient}. The non-commutative Quot scheme $ \CM^{\nc}_{\overline{r}, n}$ is realised as the moduli space of (framed) representations of the four-loop quiver (cf.~Fig.~\ref{fig:tetrahedron-quiver}) via GIT.  Our realisation of the moduli space of tetrahedron instantons inside the non-commutative Quot scheme in Theorem \ref{thm: zero locus intro} should be seen as the representation-theoretic analogue of the \emph{symplectic reduction} techniques employed in \cite{PYZ_tetrahedron}, starting from  ADHM data (describing non-commutative instantons). Granting this analogy, the special orthogonal bundle $\CL$ (and its section) are naturally determined by some moment maps coming from gauge theory (cf. Sec. \ref{sec: PYZ}).\\

By the seminal work of Oh-Thomas \cite{OT_1} (cf.~Sec.~\ref{sec: virtual classes}), the zero-locus model in Theorem \ref{thm: zero locus intro} naturally induces a \emph{virtual fundamental class} and a \emph{virtual structure sheaf} of virtual dimension zero
   \begin{align*}
    [\CM_{\overline{r}, n}]^{\vir}&\in A_{0}\left(\CM_{\overline{r}, n}, \BZ\left[\tfrac{1}{2}\right]\right),\\
\widehat{\oO}^{\vir}_{\CM_{\overline{r}, n}}&\in K_0\left(\CM_{\overline{r}, n}, \BZ\left[\tfrac{1}{2}\right]\right).
\end{align*}
We remark that it is a priori unexpected that $\CM_{\overline{r}, n} $ has virtual dimension \emph{zero}. In fact, in all the previous applications of the construction of Oh-Thomas \cite{OT_1}, the resulting virtual dimension  is positive (cf. \cite{CKM_K_theoretic, CKM_Stable_Pairs}).
    We will use these virtual classes to define invariants of $\CM_{\overline{r}, n}$ which match the ones studied by Pomoni-Yan-Zhang \cite{PYZ_tetrahedron} in string theory. However, due to the non-properness of the moduli space $\CM_{\overline{r}, n}$, we need to define the invariants via equivariant residues, with respect to a torus action on $\CM_{\overline{r}, n}$ with proper fixed locus.
 \subsection{Partition function}
 The moduli space of tetrahedron instantons  $\CM_{\overline{r}, n}$ is endowed with an action of an algebraic torus $\TT$, %of rank $3+\sum_{i=1}^4 r_i$
  whose \emph{equivariant parameters} we denote by $t=(t_1, \dots, t_4), w=(w_{11}, \dots, w_{4r_4})$. We remark that the first four variables obey the relation $t_1t_2t_3t_4=1$,  a technical assumption required at the very heart of the constructions in  \cite{OT_1}. We define the \emph{tetrahedron instantons partition function} as the  generating series of $K$-theoretic invariants
 \begin{align}\label{eqn: intro partition}
\CZ_{\overline{r}}(q)=\sum_{n\geq 0} q^n\cdot \chi\left(\mathcal{M}_{\overline{r}, n}, \widehat{\oO}^{\vir}_{\mathcal{M}_{\overline{r}, n}}\right)\in \frac{\BQ(t_1^{ \frac{1}{2}}, t_2^{\frac{1}{2}}, t_3^{ \frac{1}{2}}, t_4^{ \frac{1}{2}}, w^{ \frac{1}{2}})}{(t_1t_2t_3t_4-1)}[\![q]\!].
\end{align}
 The moduli spaces  $\CM_{\overline{r}, n}$ are not proper, therefore   the invariants  need to be defined $\TT$-equivariantly. The partition function \eqref{eqn: intro partition} reproduces the one of Pomoni-Yan-Zhang \cite{PYZ_tetrahedron} (cf.~Sec.~\ref{sec: intro String}), and interpolates between Donaldson-Thomas theory in dimension three and four (cf.~Sec.~\ref{sec: intro DT}). Morally, the invariants \eqref{eqn: intro partition} should be seen as the algebro-geometric analogue of the $\widehat{A}$-genus of the moduli space of tetrahedron instantons. We stress that, since the virtual dimension of the moduli space of tetrahedon instantons is zero, we do not need to consider additional insertions.

 Our second result is the explicit computation of the tetrahedron instantons partition function\footnote{As this work was being finalised, a preprint of Pomoni-Yan-Zhang \cite{pomoni2023probing} appeared, where a formula similar to the one in Theorem \ref{thm: intro main thm} is independently conjectured.
 %A similar formula to the one in Theorem \ref{thm: intro main thm} 
 %was independently displayed by E. Pomoni during their talk  at the \textit{\href{https://www.sissa.it/tpp/ws21/index.html}{XI Workshop on Geometric Correspondences
%of Gauge Theories}}, held at SISSA, Trieste, in September 2021
} For a formal variable $x$, set $[x]=x^{\frac{1}{2}}-x^{-\frac{1}{2}}$. Letting  $e_i$ be the $4$-tuple with  $1$ in the $i$-th entry and $0$ in all the others, we set $\CZ^{(i)}(q)=\CZ_{e_i}(q) $.
 \begin{theorem}[Theorem \ref{thm:factorization}, \ref{thm: explicit expression inv}]\label{thm: intro main thm}
    Let $\overline{r}=(r_1, r_2, r_3, r_4)$  and $r=\sum_{i=1}^4r_i$. We have 
    \[
    \CZ_{\overline{r}}((-1)^rq)=\Exp\left(-\frac{[t_1t_2][t_1t_3][t_2t_3]}{[t_1][t_2][t_3][t_4]}\frac{[\kappa_{\overline{r}}]}{[\kappa_{\overline{r}}^{\frac{1}{2}} q][\kappa_{\overline{r}}^{\frac{1}{2}}q^{-1}]} \right),
    \]
    where we set the weight  $\kappa_{\overline{r}}=\prod_{i=1}^4t_i^{-r_i}$ and $\Exp$ is the plethystic exponential \eqref{eqn: on ple}. Moreover, the instanton partition function admits the factorisation
    \begin{align*}
        \CZ_{\overline{r}}(q)=\prod_{i=1}^4 \prod_{l=1}^{r_i}\CZ^{(i)}\left((-1)^{r+1}q\kappa_i^{\frac{-r_i-1}{2}+l}\prod_{j=1}^4\kappa_j^{\frac{r_j\cdot\sgn(i-j)}{2}}\right),
    \end{align*}
    where we set the weights  $\kappa_i=t_i^{-1}$.
\end{theorem}
Theorem \ref{thm: intro main thm} generalizes the main result of \cite{FMR_higher_rank}, which deals with the Quot scheme of the  three-dimensional affine space. The factorisation of the invariants into \emph{rank 1 theories}  has already been observed in a similar situation for $K$-theoretic invariants \cite{FMR_higher_rank, KR_magnificient, CKM_crepant}, motivic invariants \cite{MR_nested_Quot, CR_framed_motivic, CRR_higher_rank}, in String Theory \cite{NP_colors, dZNPZ_playing_index_M_theory} and could find a conceptual answer in the recent (but unrelated) work of Feyzbakhsh-Thomas \cite{FT_0, FT_1} on generalized Donaldson-Thomas invariants, who exploit wall-crossing techniques to show the rank reduction.
\subsubsection{Strategy of the proof}
We explain here the main ideas in the proof of Theorem \ref{thm: intro main thm}.  The torus fixed locus  $\CM_{\overline{r}, n}^\TT $ is reduced, zero-dimensional and in bijection with tuples $\overline{\pi}=(\overline{\pi}_1, \dots, \overline{\pi}_4)$, where each $\overline{\pi}_i$ is an $r_i$-tuple of \emph{plane partitions}. Therefore by  Oh-Thomas  virtual localisation Theorem \cite{OT_1},  the tetrahedron instantons partition function can be expressed as
\begin{align}\label{eqn: intro loc}
 \CZ_{\overline{r}}(q)=\sum_{\overline{\pi}}(-1)^{\sigma_{\overline{\pi}}}[-\mathsf{v}_{\overline{\pi}}]\cdot q^{|\overline{\pi}|}.
\end{align}
Let us pause for a moment to explain the notation in \eqref{eqn: intro loc}. The sum runs over all tuples of plane partitions $\overline{\pi}$, and $|\overline{\pi}|$ denotes their \emph{size}. To each $\overline{\pi}$, we attach two items: a Laurent polynomial in the equivariant parameters $ \mathsf{v}_{\overline{\pi}} $ (the \emph{square root} of the virtual tangent space) and an (a priori non-explicit) sign $(-1)^{\sigma_{\overline{\pi}}} $. By Oh-Thomas virtual localisation, each tuple of plane partitions $\overline{\pi}$ contributes to the localised invariants with rational functions in the equivariant parameters, obtained by applying the operator $[\cdot]$ to the \emph{vertex term} $\mathsf{v}_{\overline{\pi}}$, but only up to a \emph{sign}. Despite this situation being very similar to the original  treatment of the topological vertex in  \cite{MNOP_1}, the presence of such a sign in our context is what makes the computation of the partition function
% extremely
hard and cumbersome. We remark that the  choice of square root $\mathsf{v}_{\overline{\pi}}$ is not canonical; rather, one should think of the pair $(\mathsf{v}_{\overline{\pi}}, (-1)^{\sigma_{\overline{\pi}}})$ to be canonical in a suitable sense, cf.~\cite{Mon_canonical_vertex}.\\

In order to prove Theorem \ref{thm: intro main thm}, we need the following two crucial technical results. The first one is of local nature and shows that for our particular choice of square root, the corresponding sign is constantly $(-1)^{\sigma_{\overline{\pi}}}=1$ (Theorem \ref{thm: correct sign body text}). The proof of these results follows by a purely combinatorial study of the vertex terms applied to  a structural result of Kool-Rennemo \cite{KR_magnificient} (cf.~Appendix \ref{sec: app}).

The second result is  of global nature, and asserts that the tetrahedron instantons  partition function does not depend on the \emph{framing parameters} $w_{ij}$, for all $i=1, \dots, 4$ and $j=1, \dots, r_i$ (Theorem \ref{thm: framinh independence}). This independence is the incarnation of a \emph{rigidity principle} and ultimately relies on the properness of the Quot-to-Chow morphism, and can already be spotted by the explicit formula for the partition function in Theorem \ref{thm: intro main thm}. We remark that this framing independence is peculiar to the $K$-theoretic invariants here considered, and fails for the \emph{elliptic} refinement proposed in Section \ref{sec: witten} (cf.~Remark \ref{rmk: failure witten}).

Granting this framing independence, we can set the framing parameters $w_{ij}$ to arbitrary values and take arbitrary limits. A suitable choice of such limits yields a  (non-trivial!) combinatorial expression of the partition function where the factorisation property becomes evident. Finally, the explicit expression in terms of the plethystic exponential is obtained by combining the factorised form of the partition function with the explicit formula for the \emph{rank 1} case, which had been initially conjectured by Nekrasov in String Theory \cite{Nek_Z_theory} and later proved by Okounkov \cite{Okounk_Lectures_K_theory}.

\subsubsection{Cohomological limit}
A second source of interesting invariants comes from integrating over the \emph{virtual fundamental class} in equivariant Chow cohomology, rather than in $K$-theory, as classically considered in \cite{MNOP_1}. We define the \emph{cohomological tetrahedron instantons partition function } as
\begin{align}\label{eqn: intro cohom}
\CZ^{\coh}_{\overline{r}}(q)=\sum_{n\geq 0}q^n\cdot \int_{[\CM_{\overline{r}, n}]^{\vir}}1 \in \frac{\BQ(s_1, s_2, s_3, s_4, v)}{(s_1+s_2+s_3+s_4)}[\![q]\!],
\end{align}
where $s_1, \dots, s_4, v_{11}, \dots, v_{4r_4}$ are the generators of $\TT$-equivariant Chow cohomology $A^\TT_*(\pt)$ and the integral makes sense only equivariantly. The partition function \eqref{eqn: intro cohom} carries less refined information with respect to its $K$-theoretic counterpart \eqref{eqn: intro partition}, and can in fact be realized as a \emph{limit} of the latter. This limiting procedure at the level of the partition functions in Supersymmetric String Theory corresponds to a suitable \emph{dimensional reduction} of the two-dimensional low energy effective theory. Indeed, starting from the system of D1-D7-branes in type IIB String Theory on $T^2\times\BC^4$ engineering the moduli space of tetrahedron instantons \cite{PYZ_tetrahedron}, one can shrink a circle in $T^2=S^1\times S^1$ and T-dualise along the small $S^1$. The original brane system in the T-dualised picture then reduces to a system of D0-D6 branes on $S^1$, corresponding to the quantum mechanics of the tetrahedron instantons quiver in Fig.~\ref{fig:tetrahedron-representation}. An additional shrinking to a point of the remaining $S^1$, followed by S-duality, leaves us with a D(-1)-D5-brane system, whose low-energy partition function is exactly \eqref{eqn: intro cohom}.

Applying such a limit,  we prove a closed formula for the partition function \eqref{eqn: intro cohom}. 
\begin{corollary}[Corollary \ref{cor: cohom}]\label{cor: intro cohom}
     Let $\overline{r}=(r_1, r_2, r_3, r_4)$  and $r=\sum_{i=1}^4r_i$. We have 
     \[\CZ^{\coh}_{\overline{r}}(q)=\mathrm{M}((-1)^rq)^{-\frac{(s_1+s_2)(s_1+s_3)(s_2+s_3)(r_1s_1+r_2s_2+r_3s_3 + r_4s_4)}{s_1s_2s_3s_4}},\]
     where $\mathrm{M}(q)$ is the MacMahon series \eqref{eqn: MacMahon}.
\end{corollary}
\subsubsection{Witten genus}
  Pomoni-Yan-Zhang \cite[Sec.~6]{PYZ_tetrahedron} provide an alternative computation of the tetrahedron instanton partition function --- again, in string theory --- from the elliptic genus of the low-energy worldvolume theory on D1-branes. Mathematically, we realize this as an \emph{elliptic refinement} of the invariants in \eqref{eqn: intro partition}, which we call \emph{virtual Witten genus} of $\CM_{\overline{r}, n}$
    \[
    \CZ^{\mathrm{ell}}_{\overline{r}}=\sum_{n\geq 0}q^n\cdot \chi\left(\CM_{\overline{r}, n}, \widehat{\oO}_{\CM_{\overline{r}, n}}^{\vir}\otimes \bigotimes_{k\geq 1} \Sym^\bullet_{p^k}\left(T^{\vir}_{\CM_{\overline{r}, n}}\right) \right)\in 
     \frac{\BQ(t_1^{ \frac{1}{2}}, t_2^{ \frac{1}{2}}, t_3^{ \frac{1}{2}}, t_4^{ \frac{1}{2}}, w^{ \frac{1}{2}})}{(t_1t_2t_3t_4-1)}[\![p,q]\!],\]
     where $T^{\vir}_{\CM_{\overline{r}, n}} $ denotes the virtual tangent bundle of $\CM_{\overline{r}, n}$ and $p$ is the \emph{elliptic parameter}. This elliptic genus is a generalization of the \emph{virtual chiral elliptic genus} of  Fasola-Monavari-Ricolfi \cite[Sec.~8.1]{FMR_higher_rank}, and appeared in a similar context in Donaldson-Thomas theory of Calabi-Yau fourfolds \cite[Def.~5.8]{Bojko_wall-crossing}.

     As already mentioned, the virtual Witten genus in general  depends on the framing parameters, thus preventing us from computing its generating series using the factorisation of the invariants. Nevertheless, under some suitable specialisations of the equivariant parameters $t_1, \dots, t_4$, we expect  that the generating series enjoys some rigidity behaviour of the elliptic parameter (see Remark \ref{rem: rigidity witten genus}).
 \subsection{Relation to Donaldson-Thomas Theory}\label{sec: intro DT}
Our description of the moduli space of tetrahedron instantons and its partition function nicely fits into the realm of Donaldson-Thomas (DT) theory. In fact, as already pointed out by Pomoni-Yan-Zhang \cite{PYZ_tetrahedron}, tetrahedron instantons interpolate between (higher rank) DT theory of $\BC^3$ and $\BC^4$, whose relevant moduli spaces are $\Quot_{\BC^3}({\oO^r, n})$ and $\Quot_{\BC^4}(\oO^r, n)$ (cf.~\cite{FMR_higher_rank, CKM_K_theoretic, KR_magnificient}). The interesting aspect of our construction is that it shares common features with both DT theories, in dimension three and four. In fact, despite $\CM_{\overline{r}, n}$ is a moduli space of sheaves on a threefold and   possesses a \emph{zero-dimensional} virtual fundamental class --- as in the classical DT theory of threefolds --- it seems that its virtual structure  cannot be defined directly with the machinery of (two-terms)  \emph{perfect obstruction theories} of Behrend-Fantechi and Li-Tian \cite{BF_normal_cone, LT_virtual_cycle}, but crucially needs a symmetric (three-term) obstruction theory as in the work of Oh-Thomas \cite{OT_1}, which was specifically designed for DT theory of Calabi-Yau fourfolds. Nevertheless, for suitable choices of the vector $\overline{r}$, the moduli space $\CM_{\overline{r}, n} $ reproduces the Quot scheme $\Quot_{\BC^3}({\oO^r, n})$, whose  virtual structure \emph{does} come from a perfect obstruction theory, induced by a global critical locus structure \cite{BR_higher_rank}. This shows that the natural class one should consider is always the \emph{twisted} virtual structure sheaf, which is natural from the point of view of DT theory of Calabi-Yau fourfolds, while in DT  theory of threefolds, one should \emph{twist} the virtual structure sheaf by hand to make the invariants more relevant to study, following the circle of ideas  initiated by Nekrasov-Okounkov \cite{NO_membranes_and_sheaves}. On the other hand, Theorem \ref{thm: intro main thm} shows that the tetrahedron instantons partition function is a specialisation of the partition function of higher rank Donaldson-Thomas invariants of $\BC^4$, whose explicit expression was proposed  by Nekrasov-Piazzalunga \cite{NP_colors} and which is conjectured to be the mother of all instanton partition functions \cite{Nek_magnificient_4}.

From the point of view of DT theory, one may ask if the zero-locus construction of Theorem \ref{thm: zero locus intro}  --- and its induced virtual structure --- could have an intrinsic deformation-theoretic interpretation, which does not require the construction of a smooth ambient space via a quiver with potential as in our work (cf.~\cite{Ric_virtual_motives_Quot}). This is actually the case when our moduli space of tetrahedron instantons reproduces the Hilbert scheme  of points $\Hilb^n(\BC^3)$ (cf.~\cite{RS_critical}, where it is shown that the perfect obstruction theory induced by being a critical locus coincides with the one coming by deformation theory). From this  deformation-theoretic point of view, it would be natural to replace $\Delta\subset \BC^4$ with an arbitrary pair  $(X, Y)$, where $X$ is a quasi-projective Calabi-Yau fourfold and $Y$ is a boundary divisor, such that each irreducible component is smooth. From this point of view, we would find it natural to expect a virtual pullback  formula to hold, generalizing the one of  Park \cite{Park_DT_pullbacks}, which would explain why the partition function of the magnificent four model \cite{NP_colors} specialises to the tetrahedron instanton partition function\footnote{Physically, this would correspond to the fact that our configuration of intersecting D6 branes should come from the annihilation of D8 and anti-D8 branes through tachyon condensation.}. Finally, suitably describing the virtual structure on $ \CM_{\overline{r}, n}$ via deformation theory could shed some light on its possible relation to \emph{degeneration phenomena}. A piece of strong evidence in this direction is that the invariants obtained on a moduli space of sheaves over $\Delta$ are reduced to a product of invariants over moduli spaces of sheaves over the irreducible components of $\Delta$, and that the virtual classes are (morally) deformation invariant\footnote{Deformation invariance would be expected  if our moduli spaces were proper, for instance, if we started with a pair $(X, Y)$ such that $X$ is a smooth \emph{projective} Calabi-Yau fourfold, and $Y$ a divisor.}. We plan to pursue these questions and investigate their interpretation in string theory in the future.

 \subsection{String Theory}\label{sec: intro String}
The moduli space of tetrahedron instantons was introduced in the context of Supersymmetric String Theory in \cite{PYZ_tetrahedron}. It describes the low energy theory of a system of D1-branes probing some general configuration of intersecting D7-branes. This in turn is interpreted as capturing instanton dynamics when codimension-two supersymmetric defects are present. On the geometric side, Supersymmetric String Theory in general, and D-branes systems specifically, have proved useful to provide a physical realisation of interesting moduli spaces and to compute several generating series of their (virtual) invariants. The point of contact between the mathematical and physical realms can be found in the study of the BPS sector of string theories. Indeed, this is a particularly amenable task on the physical side, as it consists of quantities which are known to be protected from quantum corrections and can therefore be studied non-perturbatively. Also, BPS-bound state counting can usually be interpreted in terms of counting/classification problems in geometry, and the technique of localisation often available in supersymmetric settings makes it feasible to compute relevant partition functions exactly. Here the connection to enumerative geometry is made even clearer, as supersymmetric localisation may be viewed as an infinite-dimensional analogue of equivariant localisation in different settings (eg., cohomological, cf.~\cite{AB_localization}, or K-theoretic localisation, cf.~\cite{Tho:formule_Lefschetz}).

One particular class of moduli spaces which can be obtained in physics from systems of D-branes are moduli spaces of (generalised) instantons. The ADHM construction of moduli spaces of anti-self dual connections as symplectic quotients shed light on their physical realisation either in the language of Sigma Models \cite{Witten_ADHM_Sigma_Model}, or as moduli spaces of BPS vacua of systems of D1-D5-branes in type IIB Supersymmetric String Theory \cite{Douglas_D_branes}. Similarly, a resolution of singularities of the moduli space of ideal instantons on $\BR^4$, realised by moduli of framed torsion-free sheaves on $\BP^2$ \cite{Nak_lectures_Hilb_schemes}, can be obtained from D-brane configurations in superstring theory, modelling instantons on noncommutative $\BR^4$. Partition functions of such theories are then computed explicitly (and non-perturbatively) via equivariant localisation \cite{MNS,Nek_instantons}. Several interesting generalisations of these constructions have been studied over the past years, e.g.~the construction of instantons as systems of D-branes on orbifolds or in the presence of surface defects \cite{DM_ALE,KT_chainsaw,BFT_defects,bonelli2020flags}, as well as on other four-dimensional manifolds \cite{Cherkis,Witten_TaubNut,BPT,FMT,BFMMRT}. Moreover, ADHM-type constructions can be exploited to construct instanton-like objects in higher dimensions \cite{BKS,MNS_D_Particle,HP}, thus providing a physical realisation of Donaldson-Thomas theory and its generalisation on fourfolds. The relation between Donaldson-Thomas theory on threefolds and string theoretic constructions in physics, such as quiver matrix models, instantons in orbifold or non-commutative backgrounds and elliptic generalisations, have been studied, for example, in \cite{AK_quiver_matrix_model,CSS_cohom_gauge_theory_DT,MR2835515,MR3046463,BBPT_elliptic_DT,Cir_M2_index}, and is rigorously understood in mathematics as in \cite{MNOP_1,MNOP_2,FMR_higher_rank}. On the other hand, the rather difficult task of interpreting higher-dimensional DT theory can be similarly undertaken in the context of String Theory as a system of D0-D8(-$\overline{\rm D8}$)-branes wrapping the affine four-space  \cite{Nek_Magnificent4_advances,NP_colors}. Also, it has been shown to enjoy a realisation as a Topological Field Theory localising on non-commutative $\Spin(7)$-instanton configurations on $\BR^8$, leading moreover to an ADHM-like construction \cite{BFTZ_ADHM8D}, and a natural generalisation to more general backgrounds, such as orbifolds (cf.~\cite{BFTZ_ADHM8D,CKM_crepant,SzaboTirelli-Orbifolds}). The construction of tetrahedron instantons as moduli spaces of BPS solutions in the background of intersecting D-branes also sits nicely in these lines of research. Indeed, the low-energy D-brane dynamics is captured by a gauge theory whose moduli space of vacua can be understood as the moduli space of stable representations of a suitable quiver, which is then interpreted as a Quot scheme on a singular threefold, which is suggestive of a DT-like interpretation. Also, the study of moduli spaces of instantons on threefolds with the inclusion of codimension-two defects naturally involves structures which are inherently fourfold in nature, thus interpolating between two different regimes in DT theory.

We point out that, during the very final stage of finishing this paper,  a preprint of  Pomoni-Yan-Zhang \cite{pomoni2023probing} appeared where they  independently conjecture the explicit formula for the tetrahedron instanton partition function, which we rigorously prove in Theorem \ref{thm: intro main thm}. Their   main focus is its relation with the index in $M$-theory, following the circle of ideas of Nekrasov-Okounkov \cite{NO_membranes_and_sheaves}.

\subsection*{Acknowledgments}
We are grateful to Francesca Carocci and Woonam Lim for useful discussion, to  Martijn Kool and Jørgen Rennemo for sharing an early version of their manuscript \cite{KR_magnificient}, and to Andrea Ricolfi for previous collaborations and enlightening discussions on Quot schemes over the years.
N.F. was supported by the Engineering and Physical Sciences Research Council through the grant no. EP/S003657/2.
S.M. was  supported by   the Chair of Arithmetic Geometry, EPFL.
%\subsection*{Conventions}
 \section{Moduli space of tetrahedron instantons}
 \subsection{Virtual fundamental classes}\label{sec: virtual classes}
 Recently Oh-Thomas \cite{OT_1} devised a machinery to construct virtual fundamental classes on moduli spaces of sheaves on Calabi-Yau fourfolds. We will use a toy model of their work to define a virtual fundamental class on the moduli space of tetrahedron instantons. For the construction of the virtual fundamental class à la Oh-Thomas in a more general setting, we refer to Park \cite{Park_DT_pullbacks}.
 
Following the conventions of  \cite[Sec.~2]{OT_1}, we say that a rank $2r$\footnote{For simplicity, we assume here that the rank is even; cf.~\cite[Sec.~1]{OT_1} for the definition in the general case.} vector bundle $\CE$ on a quasi-projective scheme $\CA$ is a \emph{special orthogonal bundle} (or an \emph{$\
 SO(2r, \BC)$-bundle}) if $\CE$ is endowed with a non-degenerate  quadratic pairing 
 \[q:\CE\otimes \CE \to \oO_{\CA},\]
together with a choice of trivialization  (an \emph{orientation})
 \[
 o: \det\CE\xrightarrow{\sim} \oO_\CA,
 \]
whose square\footnote{See \cite[Def. 2.1]{OT_1} for the precise sign conventions of orientations.} is $o^{\otimes 2}= \pm\det q $. We say that a subbundle $\Lambda \subset \CE$ is \emph{maximal isotropic} if $\rk \Lambda = r$ and $q|_{\Lambda}=0$. Any choice of maximal isotropic subbundle naturally induces an orientation of $\CE$. In fact, taking determinants of the  short exact sequence 
\[
0\to \Lambda \to \CE \to \Lambda^*\to 0, 
\]
yields $\det \CE\cong \det \Lambda \otimes \det \Lambda^* \cong \oO_\CA$, which gives the required isomorphism (cf.~\cite[Sec.~2.3]{OT_1} for the precise sign conventions).

We say that a section $s\in H^0(\CA, \CE)$ is \emph{isotropic} if $q(s,s)=0$. 
   Let a scheme $Z$
\[
\begin{tikzcd}
& \CE\arrow[d]\\
Z:=Z(s)\arrow[r, hook, "\iota"] &\CA\arrow[u, bend right, swap, "s"]
\end{tikzcd}
\]
be the zero locus of an isotropic section $s\in \Gamma(\CA,\CE)$, where $\CE$ is a $SO(2r, \BC)$-bundle over a smooth quasi-projective variety $\CA$. 
Then there exists an induced obstruction theory on $Z$
\begin{equation}\label{eqn: obs th}
 \begin{tikzcd}
 \big[\Omega^*_{\CA|Z}\arrow[r, "ds"] & \CE^*|_Z\arrow[d, "s"]\arrow[r,"(ds)^*"]&\Omega_{\CA|Z}\big]\arrow[d, phantom, "\parallel"]\arrow[r, phantom, "\cong"]&\BE\arrow[d]\\
    & \big[\CI/\CI^2\arrow[r, "d"]&\Omega_{\CA|Z}\big]\arrow[r, phantom, "\cong"]&\BL_Z,
\end{tikzcd}
\end{equation}
in $\derived^{[-2,0]}(Z)$, where $\CI \subset \oO_{A}$ is the ideal sheaf of the inclusion $Z \into \CA$ and $\BL_Z$ denotes the truncated cotangent complex. By the work of \cite{OT_1}, there exist a \emph{virtual fundamental class} and a \emph{virtual structure sheaf} on $Z$ satisfying
\begin{align*}
    \iota_*[Z]^{\vir}&=\sqrt{e}(\CE)\cap [\CA]\in A_{\dim \CA- r}\left(\CA, \BZ\left[\tfrac{1}{2}\right]\right),\\
\iota_*\widehat{\oO}_Z^{\vir}&=\sqrt{\mathfrak{e}}(\CE)\otimes K^{1/2}_\CA\in K_0\left(\CA, \BZ\left[\tfrac{1}{2}\right]\right),
\end{align*}
where $K^{1/2}_\CA$ is a square root of the canonical bundle of $\CA$ and $\sqrt{e}, \sqrt{\mathfrak{e}}$ are the Edidin-Graham square root Euler classes in Chow cohomology and $K$-theory (cf. \cite[Sec.~3.1, 5.1]{OT_1}). We call the integer $\dim \CA-r$ the \emph{virtual dimension} of $Z$. Exchanging the orientation with its opposite orientation results in changing the sign of the induced virtual classes.
\subsection{Quot scheme}
 Let $\overline{r}=(r_1, \dots, r_4)$ and set $r=\sum_{i=1}^4r_i$.  We defined the moduli space of \emph{tetrahedron instantons} as Grothendieck's Quot scheme
\[\CM_{\overline{r}, n}:=\Quot_\Delta(\CE_{\overline{r}},n),\]
where 
 $\Delta=\bigcup_{i=1}^4 \BC^3_i \subset \BC^4$ is  the union of the four coordinate hyperplanes of $\BC^4$ and
 \[\CE_{\overline{r}}=\bigoplus\iota_{i,*}\oO^{\oplus{r_i}}_{\BC^3_i}\]
 is a torsion sheaf on $\Delta$,  where $\iota_i: \BC^3_i\hookrightarrow \Delta$ are the inclusions of the irreducible components. We have natural closed immersions among Quot schemes
 \begin{align}\label{eqn: embedding Quot in QUot C4}
     \CM_{\overline{r}, n}\hookrightarrow \Quot_\Delta(\oO_\Delta^r,n)\hookrightarrow \Quot_{\BC^4}(\oO_{\BC^4}^r,n),
     %\hookrightarrow \mathcal{M}^{\mathrm{nc}}_{\overline{r}, n},
 \end{align}
 obtained by precomposing the quotients $[\CE_{\overline{r}}\onto Q]$ with \[\oO_{\BC^4}^{r}\onto \oO_{\Delta}^{r} \onto \CE_{\overline{r}}.\]
 Denote by $ \Sym^{n}\Delta$ the \emph{$n$-th symmetric power} of $\Delta$. Attached to $\CM_{\overline{r}, n}$ there is the \emph{Quot-to-Chow morphism}
 \[
 \rho: \CM_{\overline{r}, n}\to \Sym^{n}\Delta,
 \]
 which, at the level of closed points, sends a quotient to its 0-dimensional support, counted with multiplicity.
 \begin{example}\label{example: blow up}
     We work out explicitly the example of $\overline{r}=(1,1,0,0)$ and $n=1$. The moduli space of tetrahedron instantons $ \CM_{(1,1,0,0), 1}$ admits a locally closed stratification, obtained by taking the preimages of a locally closed stratification of $\Sym^{1}\Delta\cong \Delta$
     \begin{align*}
         \CM_{(1,1,0,0), 1}&=\rho^{-1}(\BC_1^3\setminus \BC_{12}^2)\sqcup \rho^{-1}(\BC_2^3\setminus \BC_{12}^2) \sqcup \rho^{-1}(\BC_{12}^2)\\
         &= \BC_1^3\setminus \BC_{12}^2\sqcup \BC_2^3\setminus \BC_{12}^2 \sqcup \BC_{12}^2\times \BP^1,
     \end{align*}
     where $\BC^2_{12}:=\BC_1^3\cap \BC_2^3$ is the intersection of the two coordinate hyperplanes and we used that 
     \[
     \rho^{-1}(\BC_{12}^2)\cong \Quot_{\BC^2}(\oO_{\BC^2}^2, 1)\cong \BC_{12}^2\times \BP^1,
     \]
     where the last isomorphism follows by \cite[Rem. 2.2]{MR_lissite}. Furthermore,  there is a better understanding of the geometry of $  \CM_{(1,1,0,0), 1} $ in terms of blow-ups.
  Let $Y=\BC^3_1\cup \BC^3_2\subset \BC^4 $ the union of the first two coordinate hyperplanes, and consider the blow-up of $\BC^4$ centred in $\BC^2_{12}$. Then the moduli space of tetrahedron instantons identifies with the proper transform of $Y$ 
 \[
 \begin{tikzcd}
       \CM_{(1,1,0,0), 1}\cong  p^{-1}(Y)\arrow{d}\arrow[r, hookrightarrow]& Bl_{\BC_{12}^2}\BC^4\arrow[d, "p"]\arrow[r, hookrightarrow]&\BC^4\times \BP^1\cong\Quot_{\BC^4}(\oO_{\BC^4}^2, 1)\arrow[dl]\\
      Y\arrow[r, hookrightarrow] &\BC^4&
  \end{tikzcd}
  \]
  Here, the composition of the first two horizontal maps is the closed embedding \eqref{eqn: embedding Quot in QUot C4}, and shows that the closure of the two strata $\BC_i^3\setminus \BC_{12}^2 $ in $ \CM_{(1,1,0,0), 1} $ are the strict transform of $Y$, while the stratum $  \BC_{12}^2\times \BP^1$ coincides with the exceptional divisor of the blow-up.
 \end{example}
 \begin{remark}
     The situation in Example \ref{example: blow up} readily generalizes to arbitrary $\overline{r}$ and $n$, where a stratification of $  \CM_{\overline{r}, n}$ is induced by a natural stratification of $\Sym^n\Delta $, thus generalizing the stratification described for $n=1$ in \cite[Eqn. (125)]{PYZ_tetrahedron}  via gauge theory. Similarly, we expect in the general case a description of the moduli space of tetrahedron instantons in terms of blow-ups, which  intriguingly hints at an interpretation as a  degeneration phenomenon.
 \end{remark}
 \subsection{Non-commutative Quot scheme}
Consider then the quiver $\LQ$ in Fig.~\ref{fig:4-loop-quiver} with one vertex and four loops, labelled $x_1,x_2,x_3, x_4$. In the following we will be interested in studying representations of the framed quiver $\tLQ$ in Fig.~\ref{fig:tetrahedron-quiver}.
% Consider the quiver $\LQ$ in Fig.~\ref{fig:4-loop-quiver} with one vertex and four loops, labelled $x_1,x_2,x_3, x_4$. A \emph{representation} of $\LQ$ is the datum of a left module over the free algebra $\BC\langle x_1, x_2, x_3, x_4\rangle$
% on four generators, which is the path algebra $\BC\LQ$ of $\LQ$. Let $\overline{r}=(r_1, \dots, r_4)$ and  add $\sum_{i=1}^4 r_i$ framings at the vertex of $\LQ$, thus forming the framed quiver $\tLQ$ in Fig.~\ref{fig:tetrahedron-quiver}.
\begin{figure}[ht]\contourlength{2.5pt}
\begin{subfigure}{.45\textwidth}
    \centering
    \begin{tikzpicture}\vspace{-3mm}
        \node(Gk0) at (2.5,0){$0$};
        \draw[<<->>](Gk0) to[out=-35,in=35,looseness=20]node{\contour{white}{$x_a$}} (Gk0);
    \end{tikzpicture}\vspace{-2mm}
    \caption{Four-loop quiver $\LQ$}
    \label{fig:4-loop-quiver}
\end{subfigure}
\begin{subfigure}{.45\textwidth}
    \centering\vspace{-3mm}
    \begin{tikzpicture}
    \node(F123) at (0,1){$\infty_1$};
    \node(F234) at (0,-1){$\infty_4$};
    \node(Gk0) at (2.5,0){$0$};
    \node[] at (0,0.125){$\vdots$};
    \draw[->](F123) to[]node{\contour{white}{$\varphi_1$}} (Gk0);
    \draw[->](F234) to[]node{\contour{white}{$\varphi_4$}} (Gk0);
    \draw[<<->>](Gk0) to[out=-35,in=35,looseness=15]node{\contour{white}{$x_a$}} (Gk0);
    \end{tikzpicture}
    \caption{Framed quiver $\tLQ$}
    \label{fig:tetrahedron-quiver}
\end{subfigure}
\caption{}\label{fig:quivers}
\end{figure}
 In particular, letting $\overline r=(r_1,\dots,r_4)$ be the dimension vector associated to the framing nodes $(\infty_i)_{i=1,\dots,4}$, we will study representations $X\in\Rep_{\BC}(\tLQ)$ with $\dim X=(\overline r,n)$. This amounts to assigning a tuple of vector spaces $(W_1,\dots,W_4,V)$, one to each node of $\tLQ$, in such a way that $\dim W_i=r_i$, $\dim V=n$, and homomorphisms $I_i\in\Hom(W_i,V)$, $B_a\in\End(V)$ according the the arrows of $\tLQ$. Schematically, a representation of $\tLQ$ is shown in Fig.~\ref{fig:tetrahedron-representation}. Studying framed representations of this sort is then equivalent to studying representations of a different framing of $\LQ$, shown pictorially in Fig.~\ref{fig:tetrahedron-representation-equivalent}.
 
% Let $\overline{r}=(r_1, \dots, r_4)$ and  add $\sum_{i=1}^4 r_i$ framings at the vertex of $\LQ$, thus forming the framed quiver $\tLQ$ in Fig.~\ref{fig:tetrahedron-quiver}. Fix a dimension vector $(\overline{r}, n)$, which means we associate  vector spaces of dimension $r_i$ on the nodes $\infty_i$ and a vector space $V$ of dimension $n$ on the node $0$.

% Representations of $\widetilde{L}$ with this given dimension vector are elements of the affine space of dimension $4n^2+ n\sum_{i=1}^4 r_i$
%  \[R_{\overline{r}, n}=\BC^4\otimes\End(V,V)\oplus\bigoplus_{i=1}^4\Hom(\BC^{r_i}, V). \]
% Consider the action of $\GL_n(\BC)$ on $R_{\overline{r}, n}$ by
% \[
% g\cdot (B_i, \dots, B_4, I_1, \dots, I_4)=(g B_1 g^{-1}, \dots, g B_4 g^{-1},  g I_1, \dots, g  I_4),\, B_i\in \End(V,V), I_i\in \Hom(\BC^{r_i}, V).
% \]

\begin{figure}[ht]\contourlength{2.5pt}
\begin{subfigure}{.45\textwidth}
    \centering
    \begin{tikzpicture}
    \node[FrameNode](F123) at (0,1){$r_1$};
    \node[FrameNode](F234) at (0,-1){$r_4$};
    \node[GaugeNode](Gk0) at (2.5,0){$n$};
    \node[] at (0,0.125){$\vdots$};
    \draw[->](F123) to[]node{\contour{white}{$I_1$}} (Gk0);
    \draw[->](F234) to[]node{\contour{white}{$I_4$}} (Gk0);
    \draw[<<->>](Gk0) to[out=-35,in=35,looseness=10]node{\contour{white}{$B_a$}} (Gk0);
    \end{tikzpicture}
    \caption{Representation of $\tLQ$}
    \label{fig:tetrahedron-representation}
\end{subfigure}
\begin{subfigure}{.45\textwidth}
    \centering
    \begin{tikzpicture}
    \node[FrameNode](F0) at (0,0){$1$};
    \node[GaugeNode](Gk0) at (2.5,0){$n$};
    \node[] at (1.25,0.125){$\vdots$};
    \draw[->](F0) to[out=50, in=130]node{\contour{white}{$I_{1}^{(1)}$}} (Gk0);
    \draw[->](F0) to[out=-50, in=230]node{\contour{white}{$I_{r_4}^{(4)}$}} (Gk0);
    \draw[<<->>](Gk0) to[out=-35,in=35,looseness=10]node{\contour{white}{$B_a$}} (Gk0);
    \end{tikzpicture}
    \caption{Equivalent quiver representation}
    \label{fig:tetrahedron-representation-equivalent}
\end{subfigure}
\caption{}\label{fig:representations}
\end{figure}

In the following, we will often use the notation $r=r_1+\cdots+r_4$. Let then $V$ be a given $n$-dimensional vector space, and consider the affine $(4n^2+nr)$-dimensional representation space
\[
\mathsf R_{\overline r,n}\defeq \BC^4\otimes\End(V)\oplus\Hom(\BC,V)^{\oplus r}.
\]
An element $X\in\mathsf R_{\overline r,n}$, together with the choice of $V$, forms a representation as in  Fig.~\ref{fig:tetrahedron-representation-equivalent}. In the following, we will refer to $X\in\mathsf R_{\overline r,n}$ itself as a representation. Also, a representation $X\in\mathsf R_{\overline r,n}$ induces a representation as in Fig.~\ref{fig:tetrahedron-representation} by defining
\[I_i\left(\overbrace{0,\dots,0}^{j-1},1,\overbrace{0,\dots,0}^{r_i-j}\right)=I_i^{(j)}(1),\]
and $W_i=\BC^{r_i}$. Conversely, any representation $(W_1,\dots,W_4,V,B_1,\dots,B_4,I_1,\dots,I_4)$ as in Fig.~\ref{fig:tetrahedron-representation} induces a representation as in Fig.~\ref{fig:tetrahedron-representation-equivalent} by choosing a basis $\{e^{(i)}_j\}\subset W_i$ and setting
\[I_i^{(j)}(1)=I_i\left(e_j^{(i)}\right).\]

The group $\GL(V)$ acts naturally on the representation space by simultaneously conjugating the endomorphisms and scaling the vectors. Consider then then the open subscheme $\mathsf U_{\overline r,n}\subset\mathsf R_{\overline r,n}$ parametrising those quiver representations $(B_1,\dots,B_4,I_1^{(1)},\dots,I_4^{(r_4)})$ such that the morphisms $(I_i^{(j)})_{i,j}$ \textit{jointly generate} $V\cong\BC^n$ as a $\BC\langle B_1,B_2,B_3,B_4\rangle$-module, that is
\begin{equation}\label{eqn:stability-vectors-generate-V}
    \bigoplus_{i=1}^4\bigoplus_{j=1}^{r_i}\BC\langle B_1,B_2,B_3,B_4\rangle I_i^{(j)}(1)\cong V.
\end{equation}
The $\GL(V)$-action is free on $\mathsf U_{\overline r,n}$, so the quotient $\mathcal M^{\nc}_{\overline r,n}\defeq\mathsf U_{\overline r,n}/\GL(V)$ exists as a smooth $(3n^2+nr)$-dimensional quasiprojective variety, cf.~\cite[Lemma~2.1]{Ric_noncomm}. The moduli space of representations thus constructed is known as the \textit{non-commutative Quot scheme}, and parametrises quotients $[\mathbb C\langle x_1,\dots,x_4\rangle^{\oplus r}\onto Q]$  of the free rank $r$ algebra $\mathbb C\langle x_1,\dots,x_4\rangle^{\oplus r}$ on four generators, cf.~\cite[Thm.~2.5]{BR_higher_rank} and \cite[Prop.~2.3]{Ric_noncomm}, and can be furthermore identified with the GIT quotient of $\mathsf R_{\overline r,n} $ with respect to King stability \cite{King_stability}.

 \subsection{Zero-locus construction}\label{sec: zero locus}
We prove in this section Theorem \ref{thm: zero locus intro} from the Introduction. Over the affine space $\mathsf R_{\overline r,n}$,  consider the trivial  vector bundle of rank $6n^2+2rn$

\[
\mathcal L\defeq\left(\Lambda^2\BC^4\otimes\End(V)\right)\oplus\bigoplus_{i=1}^4\left(\Hom(W_i,V)\oplus\Hom(W_i,V)^\ast\right).
\]
We endow $\CL$ with a non-degenerate quadratic pairing $q:\mathcal L^{\otimes 2}\to\oO_{\mathsf R_{\overline r,n}}$, defined fiberwise by
\[
q\left((\phi_1\otimes f_1,\boldsymbol w_1,\boldsymbol\omega_1),(\phi_2\otimes f_2,\boldsymbol w_2,\boldsymbol\omega_2)\right)=\left(\phi_1\wedge\phi_2\right)\otimes\tr(f_1\circ f_2)+\frac{1}{2}\left(\boldsymbol\omega_1(\boldsymbol w_2)+\boldsymbol\omega_2(\boldsymbol w_1)\right),
\]
where  $\phi_i\in\Lambda^2\BC^4$, $f_i\in\End V$, $\boldsymbol w_i\in\bigoplus_{j=1}^4\Hom(W_j,V)$, $\boldsymbol \omega_i\in\bigoplus_{j=1}^4\Hom(W_j,V)^\ast$. The quadratic bundle $(\CL, q)$ restricts to the open locus $\mathsf U_{\overline r,n}\subset \mathsf R_{\overline r,n}$ and is naturally $\GL(V)$-equivariant, with respect to the $\GL(V)$-action on $ U_{\overline r,n}$. Therefore, $(\CL, q)$ descends to a vector bundle with a quadratic pairing on the quotient $\mathcal M^{\nc}_{\overline r,n}$. With a slight abuse of notation, we keep denoting the induced vector bundle by $(\mathcal L,q)$.
%This vector bundle (with its quadratic pairing) appears in a similar context in the  work of Kool-Rennemo \cite{KR_magnificient}.

 \begin{theorem}\label{thm: isotropic construction}
    Let $\overline{r}=(r_1, \dots, r_4)$ and $n$ be  non-negative integers. There exists an isotropic section $s\in H^0(\CM_{\overline r,n}^{\nc},\CL)$ such that $\CM_{\overline r,n}$ is realised as the zero locus $Z(s)$
    \[
\begin{tikzcd}
& \CL\arrow[d]\\
\CM_{\overline{r}, n}\cong Z(s)\arrow[r, hook, "\iota"] &\CM^{\nc}_{\overline{r}, n}.\arrow[u, bend right, swap, "s"]
\end{tikzcd}
\]
\end{theorem}

\begin{proof}
We start by constructing an isotropic section of $(\CL,q)$ on the affine space $R_{\overline r,n}$.  Let $X=(B_1, \dots, B_4, I_1, \dots, I_4)\in R_{\overline r,n}$ be a quiver representation and  let $(e_i)_{i=1, \dots, 4}$ be a basis of $\BC^4$. Define the section 
\begin{align*}
    s(X)=\sum_{a,b=1}^4(e_a\wedge e_b)\otimes B_aB_b+\sum_{i=1}^4 B_iI_i.
\end{align*}
In particular, the restriction of the section to the last summand of $\CL$ vanishes
\begin{align*}
    s|_{\bigoplus_{i=1}^4\Hom(W_i,V)^\ast}=0.
\end{align*}
The section $s$ is clearly isotropic with respect to the quadratic pairing $q$, i.e. $q(s,s)=0$. 
Since the restriction of the section to the open locus $ \mathsf U_{\overline r,n}$ is $\GL(V)$-equivariant, it descends to an isotropic section on the quotient $ \CM^{\nc}_{\overline r,n}$, which  we keep denoting by  $s\in H^0(\CM^{\nc}_{\overline r,n},\CL)$.

There is a natural closed embedding
\begin{align*}
    \iota:\CM_{\overline r,n}\hookrightarrow \CM^{\nc}_{\overline r,n},
\end{align*}
which is obtained by the closed embedding \eqref{eqn: embedding Quot in QUot C4}, followed by the closed embedding of the Quot scheme $\Quot_{\BC^4}(\oO_{\BC^4}^r,n)$ into its non-commutative counterpart, defined via their moduli functors. We claim that
\[\CM_{\overline r,n}\cong Z(s).\]
Let $[\alpha: \CE_{\overline{r}}\onto Q]\in  \CM_{\overline r,n}$ be a closed point. Recall that the embedding into the non-commutative Quot scheme identifies this quotient with the quiver representation $(B_1, \dots, B_4, I_1, \dots, I_4)\in  \CM^{\nc}_{\overline r,n}$, where 
\begin{align*}
    &B_i:Q\to Q, \quad i=1, \dots, 4,\\
    &I^{(j)}_i=\alpha(0,\dots, 0, 1, 0, \dots, 0)\in Q, \quad i=1, \dots, 4,\, j=1, \dots, r_i.
\end{align*}
Via this correspondence, the matrices $B_i$ are the module multiplication by $x_i$, and the cyclic vectors $I^{(j)}_i$ are the images of the constant 1, each coming from a summand of $\CE_{\overline{r}}$. 

By construction, if $(B_1, \dots, B_4, I_1, \dots, I_4)\in  \CM^{\nc}_{\overline r,n} $ is a closed point in the image of $\iota$, then it satisfies the relations
\begin{align}\label{eqn: relations}
\begin{split}
    &B_aB_b-B_bB_a=0, \quad a,b=1, \dots, 4,\\
&B_iI_i=0, \quad i=1, \dots, 4.
\end{split}
\end{align}
On the other hand, consider a quotient  $[\mathbb C\langle x_1,\dots,x_4\rangle^{\oplus r}\onto Q] \in \CM^{\nc}_{\overline r,n}$ satisfying the relations \eqref{eqn: relations}. Then the quotient factors through
\[
\begin{tikzcd}
    \mathbb C\langle x_1,\dots,x_4\rangle^{\oplus r} \arrow[r,two heads]\arrow[d,two heads]& Q\\
  \CE_{\overline{r}} \arrow[ur, two heads]&   
\end{tikzcd}
\]
thus defining a closed point in $\CM_{\overline r,n}$, as desired.
We conclude by noticing that  the same argument carries over flat families, which yields the desired isomorphism.
%In other words, we have a bijective correspondence  $ \CM_{\overline r,n}=Z(s)$ at the level of closed points. To prove this bijection comes from an isomorphism of schemes, we employ the same strategy as in \cite[Thm. 2.8]{Mon_double_nested} (cf. also \cite[Lemma 3.3]{GT_nested_Hilb1}).
\end{proof}
 \subsection{The maximal isotropic subbundle}
 We endow the vector bundle $\CL$ with a structure of special orthogonal bundle, by showing it admits a maximal isotropic subbundle. Consider the subbundles
 \begin{align*}
     \tilde{\Lambda}&:=\left(\langle e_4\rangle\wedge \BC^3\right)\otimes \End(V)\subset \left(\Lambda^2\BC^4\otimes\End(V)\right)\\
     \Lambda&=\tilde{\Lambda}\oplus \bigoplus_{i=1}^4\Hom(W_i,V)^*\subset \CL,
 \end{align*}
 on the affine space $\mathsf R_{\overline r,n}$. Then $\Lambda$ is a maximal isotropic subbundle. In fact, $\Lambda$ has rank $3n^2+ rn$ and the restriction of the quadratic pairing vanishes.  The maximal isotropic subbundles $\tilde{\Lambda},\Lambda$ descend to  maximal isotropic subbundles on the non-commutative Quot scheme $\CM^{\nc}_{\overline r,n} $, which we keep denoting by $\tilde{\Lambda}, \Lambda$.
 \begin{remark}
        The non-commutative Quot scheme $\CM^{\nc}_{\overline r,n} $ and a summand of its special orthogonal bundle $\CL$  (with its maximal isotropic subbundle $\tilde{\Lambda}$) appear also in the work of  Kool-Rennemo \cite{KR_magnificient}, who exploit them to study Donaldson-Thomas theory of $\BC^4$.
    \end{remark}

Since $\CL$ is a special orthogonal bundle,  Theorem \ref{thm: isotropic construction} immediately implies the following corollary.
\begin{corollary}\label{cor: virtual classes}
   The moduli space of tetrahedron instantons  $\CM_{\overline{r}, n}$ is endowed with a virtual fundamental class and a virtual structure sheaf of virtual dimension zero
   \begin{align*}
    [\CM_{\overline{r}, n}]^{\vir}&\in A_{0}\left(\CM_{\overline{r}, n}, \BZ\left[\tfrac{1}{2}\right]\right),\\
\widehat{\oO}^{\vir}_{\CM_{\overline{r}, n}}&\in K_0\left(\CM_{\overline{r}, n}, \BZ\left[\tfrac{1}{2}\right]\right).
\end{align*}
\end{corollary}
    We will use these virtual classes to define invariants of $\CM_{\overline{r}, n}$ which match the ones studied by Pomoni-Yan-Zhang \cite{PYZ_tetrahedron} in string theory.

 \section{The partition function}
\subsection{Torus representations and their weights}
Let $\TT = (\BC^*)^g$ be an algebraic torus, with character lattice $\widehat{\TT} = \Hom(\TT,\BC^\ast) \cong \BZ^g$. Let $K^0_{\TT}(\pt)$ be the Grothendieck group of the category of $\TT$-representations. Any finite-dimensional $\TT$-representation $V$ splits as a sum of $1$-dimensional representations called the \emph{weights} of $V$. Each weight corresponds to a character $\mu \in \widehat{\TT}$, and in turn each character corresponds to a monomial $\tf^\mu = \tf_1^{\mu_1}\cdots \tf_g^{\mu_g}$ in the coordinates of $\TT$. In other words, there is an isomorphism 
\begin{equation*}
K^0_{\TT}(\pt) \cong \BZ \left[\tf^\mu \mid \mu \in \widehat{\TT}\right],
\end{equation*}
identifying the class of a $\TT$-representation with  its decomposition into weight spaces. We will therefore sometimes identify a (virtual) $\TT$-representation with its character. With a slight abuse of notation, we still denote by $\TT$ any finite cover of the algebraic torus $ \TT$ where the square roots $t^{\frac{\mu}{2}}$ of some irreducible characters $t^\mu$ are well-defined. This will be needed as most of our constructions need naturally  to extract square roots.
%If $X$ is a scheme with a trivial $\TT$-action, every $\TT$-equivariant coherent sheaf on $X$ decomposes as $F=\bigoplus_{\mu\in \widehat{\TT}}F_\mu\otimes \tf^\mu $.
\subsection{Torus action}
Let $\overline{r}=(r_1, \dots, r_4)$ and $r=\sum_{i=0}^4 r_i$ and define the algebraic tori
\[
\TT=\TT_0\times \TT_1, \quad \TT_0=\{t_1t_2t_3t_4=1\}\subset (\BC^*)^4, \quad \TT_1=(\BC^*)^{r}.
\]
The torus $\TT_0$ acts on $\BC^4$ by
\[(t_1,\dots, t_4)\cdot (x_1, \dots, x_4)=(t_1x_1,\dots, t_4 x_4)\]
 preserving the Calabi-Yau volume form. In particular, it restricts to the union of the coordinate hyperplanes $\Delta\subset \BC^4$ and naturally lifts to the moduli space of tetrahedron instantons $\CM_{\overline{r}, n}$, by moving the support of the quotients. The torus $\TT_1$  acts on $\CM_{\overline{r}, n}$ by scaling the fibers of $\CE_{\overline{r}}$. Therefore, we have an induced $\TT$-action on $\CM_{\overline{r}, n}$. Equivalently, the  $\TT$-action   on $ \CM_{\overline{r}, n}$ is the restriction of the  $\TT$-action on the moduli space of non-commutative Quot scheme $ \CM^{\mathrm{nc}}_{\overline{r}, n}$ given on points by
 \[(t, w)\cdot (B_1, \dots, B_4,  I^{(1)}_1, \dots, I_4^{(r_4)})=(t_1^{-1} B_1, \dots, t_4^{-1} B_4, w_{11}^{-1}I_1^{(1)}, \dots, w_{4r_4}^{-1}I^{(r_4)}_{4}), \]
 where $w_{ij}$ are the coordinates of $\TT_1$, for $i=1, \dots, 4$ and $j=1\dots, r_i$.
 
\begin{remark}\label{lemma:compact_fixed_locus_Quot}
The fixed locus $\CM_{\overline{r}, n}^{\TT_0}$ is proper. Indeed, a $\TT_0$-invariant surjection $\CE_{\overline{r}}\onto Q$ necessarily has the quotient $Q$  entirely supported at the origin $0\in \Delta$. Hence
\[
\CM_{\overline{r}, n}^{\TT_0} \into \rho^{-1}(n\cdot[0])
\]
sits  as a closed subscheme inside the fiber over $n\cdot [0]$ of the Quot-to-Chow morphism $ \rho:\CM_{\overline{r}, n}\to \Sym^n\Delta$.
 But $\rho^{-1}(n\cdot[0])$ is proper, since $\rho$ is a proper morphism.
\end{remark}
 
The torus fixed locus of the moduli space of tetrahedron instantons is classified in terms of \emph{plane partitions}.
\begin{definition}
    A \emph{plane partition} $\pi$ is a finite subset of $\BZ^3_{\geq 0}$ such that if $(i,j,k)\in \pi$ with $(i,j,k)\neq 0$, then $(i-1, j,k), (i,j-1,k), (i,j,k-1)\in \pi$, and we denote by $|\pi|$ its \emph{size}. Alternatively, we can view the plane partition  as boxes stacked in $\BR^3$, where the lattice points are determined by the \emph{interior} corners of the boxes, the corners closest to the origin. For a tuple of plane partitions $\overline{\pi}=(\pi_1, \dots, \pi_r)$ we denote its \emph{size} by $|\overline{\pi}|=\sum_{i=1}^r |\pi_i|$.
    %A \emph{framed plane partition} is a tuple $\overline{\pi}=(\pi_1,\dots \pi_r)$ of 
\end{definition}
It is well-known that plane partitions are in bijection with monomial ideals of $\BC[x,y,z]$
\[\pi\subset \BZ^3_{\geq 0} \longleftrightarrow I_{\pi}=(x^iy^jz^k)_{(i,j,k)\notin \pi}\subset \BC[x,y,z],\]
where the lattice points of $\BZ^3_{\geq 0}$ correspond to monomials in three variables (see e.g. \cite[Sec.~3.4]{FMR_higher_rank}).
    \begin{prop}\label{prop: fixed locus reduced}
    Let $\overline{r}=(r_1,\dots, r_4)$. The $\TT$-fixed locus $\mathcal{M}_{\overline{r}, n}^\TT$ is reduced, zero-dimensional and in bijection with
    \[\left\{\overline{\pi}=(\overline{\pi}_1, \dots, \overline{\pi}_4): \overline{\pi}_i=(\pi_{il})_{1\leq l\leq r_4}, \mbox{$\pi_{il}$ plane partitions such that } |\overline{\pi}|=n    \right\}.\]
    \end{prop}
    \begin{proof}
        %The closed embedding $\mathcal{M}_{\overline{r}, n}\hookrightarrow \Quot_{X}(\oO_X^{r}, n)$ is naturally $\TT$-equivariant, where the $\TT$-action on $\Quot_{X}(\oO_X^{r}, n)$ is defined analogously.
        By  the same arguments as in \cite{Bifet} its $\TT_1$-fixed locus is scheme-theoretically isomorphic to 
        \[
\mathcal{M}_{\overline{r}, n}^{\TT_1}\cong \bigsqcup_{n=\sum_{i=1}^4\sum_{j=1}^{r_i}n_{ij}}\prod_{i=1}^4\prod_{j=1}^{r_i}\Quot_{\Delta}(\iota_{i, *}\oO_{\BC^3_i},n_{ij}).
        \]
    Moreover, we have an identification of each factor with a Hilbert scheme of points \[\Quot_{\Delta}(\iota_{i, *}\oO_{\BC^3_i},n_{ij})\cong \Hilb^{n_{ij}}(\BC_i^3), \] such that  the natural $(\BC^*)^3$-action on $\BC^3_i$ lifts on $\Hilb^{n_{ij}}(\BC_i^3)$ and coincides with the induced $\TT_0$-action on the left-hand-side.
The $ (\BC^*)^3$-fixed locus $  \Hilb^{n_{ij}}(\BC_i^3)^{(\BC^*)^3}$ is reduced and zero-dimensional, which implies that the same holds for 
 $\mathcal{M}_{\overline{r}, n}^\TT$ and that every $\TT$-fixed point in $ \mathcal{M}_{\overline{r}, n}^\TT$ corresponds to a quotient
        \[\bigoplus_{i=1}^4\bigoplus_{j=1}^{r_i}I_{ij}\hookrightarrow \bigoplus_{i=1}^4 \iota_{i,*}\oO_{\BC^3_i}^{r_i}\onto \bigoplus_{i=1}^4\bigoplus_{j=1}^{r_i}\oO_{Z_{ij}},\]
        where each $Z_{ij}\subset \BC^3_{i}$ is a closed zero-dimensional subscheme with monomial ideal sheaf $I_{ij}$, each of which corresponds to a plane partition. Reversing the correspondence concludes the argument.
    \end{proof}
     Define the \emph{MacMahon series}
\begin{align}\label{eqn: MacMahon}
    \mathrm{M}(q)=\prod_{n\geq 1}\frac{1}{(1-q^n)^n}.
\end{align}
    \begin{corollary}
    Let $\overline{r}=(r_1, \dots, r_4)$ and $r=\sum_{i=1}^4r_i$.    The generating series of topological Euler characteristics is
        \[
        \sum_{n\geq 0}e(\mathcal{M}_{\overline{r}, n})     \cdot q^n =   \mathrm{M}(q)^r.
        \]
    \end{corollary}
    \begin{proof}
        The Euler characteristic of a scheme coincides with the one of its $\TT$-fixed locus, therefore the result follows by Proposition \ref{prop: fixed locus reduced} combined with the classical fact that
        \[
        \sum_{\pi}q^{|\pi|}= \mathrm{M}(q),
        \]
        where the sum is over all plane partitions.
    \end{proof}
\subsection{Invariants}\label{sec:invariants}
For any  $\overline{r}=(r_1,\dots, r_4)$ and $n\in \BZ_{\geq 0}$, denote by $t_1,\dots, t_4$ the irreducible characters of the $\TT_0$-action on $\mathcal{M}_{\overline{r}, n}$ and by $w_{ij}$ the irreducible characters of the $\TT_1$-action, where $i=1,\dots, 4$ and $j=1,\dots r_i$. Since all ingredients in Theorem \ref{thm: isotropic construction} are naturally $\TT$-equivariant and in particular the $\TT$-action preserves the quadratic pairing\footnote{The requirement to restrict to the Calabi-Yau condition $t_1t_2t_3t_4=1$ is needed precisely to preserve the quadratic pairing equivariantly.} on $\CL$, the virtual structure sheaf $ \widehat{\oO}^{\vir}_{\mathcal{M}_{\overline{r}, n}}$ naturally lifts to a $\TT$-equivariant class in equivariant $K$-theory. We define the $\TT$-equivariant $K$-theoretic invariants
\begin{align}\label{eqn: invariants}
\CZ_{\overline{r},n}=\chi\left(\mathcal{M}_{\overline{r}, n}, \widehat{\oO}^{\vir}_{\mathcal{M}_{\overline{r}, n}}\right)\in \frac{\BQ(t_1^{ \frac{1}{2}}, t_2^{\frac{1}{2}}, t_3^{ \frac{1}{2}}, t_4^{ \frac{1}{2}}, w^{ \frac{1}{2}})}{(t_1t_2t_3t_4-1)}.
\end{align}
Since $\CM_{\overline{r}, n}$ is not proper, one cannot directly define $K$-theoretic invariants via proper push-forward in $K$-theory. Instead,  since the fixed locus $\CM_{\overline{r}, n}^\TT $ is proper, we define invariants as the composition
\begin{align}\label{def: inv by loc def}
  \chi\left(\mathcal{M}_{\overline{r}, n}, \cdot\right):  K^{\TT}_0\left(\CM_{\overline{r}, n}, \BZ\left[\tfrac{1}{2}\right]\right)\to K^{\TT}_0(\CM_{\overline{r}, n})_{\mathrm{loc}}\xrightarrow{\sim} K^{\TT}_0(\CM^{\TT}_{\overline{r}, n})_{\mathrm{loc}}\to K_0^\TT(\pt)_{\mathrm{loc}}.
\end{align}
Here, the first map is a suitable localisation of $ K^{\TT}_0\left(\CM_{\overline{r}, n}, \BZ\left[\tfrac{1}{2}\right]\right)$, the second map is  Thomason's abstract localisation \cite{Tho:formule_Lefschetz} and the third map is the proper pushforward on  $\CM_{\overline{r}, n}^\TT $ (extended to the localisation).

We define the \emph{tetrahedron instanton partition function} as the generating series
\begin{align*}
    \CZ_{\overline{r}}(q)=\sum_{n\geq 0} \CZ_{\overline{r},n}\cdot q^n \in \frac{\BQ(t_1^{ \frac{1}{2}}, t_2^{ \frac{1}{2}}, t_3^{ \frac{1}{2}}, t_4^{ \frac{1}{2}}, w^{ \frac{1}{2}})}{(t_1t_2t_3t_4-1)}[\![q]\!].
\end{align*}
Let $T_{\CM_{\overline{r}, n}}^{\vir}$ be the \emph{virtual tangent bundle} of $\mathcal{M}_{\overline{r}, n}$, that is  the dual of the obstruction theory $\BE$ in \eqref{eqn: obs th}
\begin{align*}
    T_{\CM_{\overline{r}, n}}^{\vir}=[{T_{\CA}}|_{\mathcal{M}_{\overline{r}, n}}\to \CL|_{\mathcal{M}_{\overline{r}, n}}\to {T^*_{\CA}}|_{\mathcal{M}_{\overline{r}, n}}],
\end{align*}
and denote by $N^{\vir}$ the \emph{virtual normal bundle}, which is  the $\TT$-movable part of $T_{\CM_{\overline{r}, n}}^{\vir}$. We will denote by $ T_{\CM_{\overline{r}, n}}^{\vir},  N^{\vir}$ also their classes in $K$-theory, whenever it is  clear from the context.

By Oh-Thomas localisation theorem in K-theory  \cite{OT_1} we can compute the invariants via $\TT$-equivariant residues on the fixed locus
\begin{align*}\label{eqn: loc K}
    \chi\left(\mathcal{M}_{\overline{r}, n}, \widehat{\oO}^{\vir}_{\mathcal{M}_{\overline{r}, n}}\right)=\chi\left(\mathcal{M}_{\overline{r}, n}^\TT,\frac{ \widehat{\oO}^{\vir}_{\mathcal{M}_{\overline{r}, n}^\TT}}{\sqrt{\mathfrak{e}}(N^{\vir}) } \right),
\end{align*}
where $\mathfrak{e}$ is the ($\TT$-equivariant) $K$-theoretic Euler class, which is defined as follows. Let  $\CE$  be a $\TT$-equivariant locally free sheaf on $\CM_{\overline{r}, n}$. We define
\[\mathfrak{e}(\CE):= \Lambda^\bullet \CE^*=\sum_{i\geq 0} (-1)^i \Lambda^i \CE^*\in K^0_{\TT}(\CM_{\overline{r}, n}),\]
and extend it  by linearity to any class in $ K^0_\TT(\CM_{\overline{r}, n})$. Finally, $ \sqrt{\mathfrak{e}}(\cdot)$ is the ($\TT$-equivariant) $K$-theoretic square-root Euler  class; the complete description and construction of this class is in \cite[Sec.~5.1]{OT_1}. We remark that in the proof  of the localisation theorem,  Oh-Thomas \cite{OT_1}  prove  that  the fixed locus $ \mathcal{M}_{\overline{r}, n}^\TT$ is naturally endowed with an \emph{orientation} (in the sense of \cite{CGJ_orientability}) -- and therefore a virtual structure sheaf -- and that the virtual normal bundle $N^{\vir}$ is a self-dual complex (with a natural induced orientation), which implies that its square-root Euler class is well-defined (cf.~\cite[Eqn. (112)]{OT_1}). By Proposition \ref{prop: fixed locus reduced} the fixed locus $\mathcal{M}_{\overline{r}, n}^\TT$  consists of  finitely many reduced  points labelled by tuples of plane partitions. Thus, if we let $Z\in \mathcal{M}_{\overline{r}, n}^\TT$ be a fixed point, the induced virtual fundamental class  satisfies $[Z]^{\vir}=\pm [Z]$, for a well-determined sign. Therefore the tetrahedron partition functions is given by
\begin{align*}
     \CZ_{\overline{r}}(q)=\sum_{\overline{\pi}}(-1)^{o_{\overline{\pi}}}\sqrt{\mathfrak{e}}(-T^{\vir}_{\overline{\pi}})\cdot q^{|\overline{\pi}|},
\end{align*}
where the sum runs over tuples $ \overline{\pi}=(\overline{\pi}_1, \dots, \overline{\pi}_4)$, where each $\overline{\pi}_i$ is an $r_i$-tuple of plane partitions,  $T^{\vir}_{\overline{\pi}}$ is the virtual tangent bundle at the fixed point corresponding to $\overline{\pi}$ via the correspondence in Proposition \ref{prop: fixed locus reduced} and $(-1)^{o_{\overline{\pi}}}$ is the sign of the induced virtual fundamental class (with respect to the standard fundamental class). We used the fact that the virtual tangent bundle is $\TT$-movable by Proposition \ref{prop: t movable}.

The definition of the square-root Euler  class is a priori non-explicit and depends  on the chosen orientation of  $ \mathcal{M}_{\overline{r}, n}$. To ease our life, we define the \emph{square root} of a $\TT$-representation.
\begin{definition}[{\cite[Def. 1.2]{Mon_canonical_vertex}}]\label{def: square root}
Let $V\in K^0_\TT(\pt)$ be a virtual $\TT$-representation. We say that $T\in K^0_\TT(\pt)$ is a \emph{square root} of $V$ if
\[
V=T+\overline{T}\in K_\TT^0(\pt),
\]
where $\overline{(\cdot)}$ denotes the dual $\TT$-representation.
\end{definition}
 Let $ V$ be a $\TT$-representation with a   square root  $T$ in $K^0_\TT(\pt)$. Its $K$-theoretic square root Euler class satisfies
%\footnote{In more generality,  one should replace the square root $T$ with a maximal isotropic subbundle $\Lambda$ of $V$ (cf. \cite[Def. 5.3, Prop. 5.4]{OT_1}).}
\begin{align}\label{eqn: square root euler class}
    \sqrt{\mathfrak{e}}(V)&=\pm \mathfrak{e}(T)\otimes ({\det}T)^{\frac{1}{2}} \in K^0_\TT\left(\pt, \BZ\left[\tfrac{1}{2}\right]\right),
\end{align}
for a well-defined (but non-explicit) sign.
For an irreducible $\TT$-representation $t^\mu$,  define
\[
[t^\mu]=t^{\frac{\mu}{2}}-t^{-\frac{\mu}{2}}\in  K^0_\TT(\pt)
\]
and extend it by linearity to any  $T\in K_\TT^0(\pt)$ by setting $[t^{\mu}\pm t^{\nu}]=[t^{\mu}][t^{\nu}]^{\pm 1}$, as long as $\nu$ is not the trivial weight. It is shown in \cite[Sec.~6.1]{FMR_higher_rank} that 
\begin{align*}
   \mathfrak{e}(T)\otimes ({\det}T)^{\tfrac{1}{2}}=[T], 
\end{align*}
for any virtual $\TT$-representation $T\in K_\TT^0(\pt)$. Therefore we can express \eqref{eqn: square root euler class} explicitly as
\begin{align}\label{eqn: square roots with brackets 2}
     \sqrt{\mathfrak{e}}(V)= \pm [T].
\end{align}
We derive in the next section a vertex formalism that computes the partition function of tetrahedron instantons, by choosing suitable square roots of the virtual tangent bundle and the associated signs.
\subsection{Vertex formalism}
Let $r=(r_1, \dots r_4)$ and $\overline{\pi}=(\overline{\pi}_1, \dots, \overline{\pi}_4)$ correspond to a $\TT$-fixed quotient   $[\CE_{\overline{r}}\onto Q_{\overline{\pi}}]\in \CM_{\overline{r}, n}^\TT$, where each $ \overline{\pi}_i=(\pi_{i1}, \dots, \pi_{ir_i})$ is a tuple of plane partitions. We denote by $T^{\vir}_{\overline{\pi}}$ the virtual tangent space at the fixed point $\overline{\pi}$. Recall that $Q_{\overline{\pi}}$ decomposes as 
\[Q_{\overline{\pi}}=\bigoplus_{i=1}^4\bigoplus_{j=1}^{r_i}\oO_{Z_{{\pi_{ij}}}}, \]
where each $Z_{\pi_{ij}}\subset\BC^3_i$ is a zero-dimensional closed subscheme defined by a monomial ideal corresponding to a plane partition $\pi_{ij}$. We keep denoting their  global sections by $Z_{\pi_{ij}}$ (resp. $Q_{\overline{\pi}}$)  which  are given, as a $\TT_0$ (resp. $\TT$)-representation, by
\begin{align}\label{eqn: decomposition of Q}
\begin{split}
    Z_{\pi_{ij}}&=\sum_{(a,b,c)\in \pi_{ij}}t_{i_1}^at_{i_2}^bt_{i_3}^c,   \\
Q_{\overline{\pi}_i}&=\sum_{j=1}^{r_i}w_{ij}Z_{\pi_{ij}}\\
Q_{\overline{\pi}}&=\sum_{i=1}^4Q_{\overline{\pi}_i},
\end{split}
\end{align}
where $\{i_1, i_2, i_3\}$ are the three indices in $\{1,2,3,4\}$ different from $i$. We denote by $\overline{(\cdot)}$ the involution on $K^0_\TT(\pt)$ which acts as $\overline{t^\mu}=t^{-\mu}$ on irreducible representations and  by $P_I=\prod_{l\in I}(1-t_l)$, for any set of indices $I$.
%and by $\kappa_i=t_{i_1}t_{i_2}t_{i_3}$.
For $i=1,\dots, 4$ set the $\TT_1$-representations
\begin{align*}
    K_i&=\sum_{j=1}^{r_i}w_{ij},\\
    K&=\sum_{i=1}^4 K_i.
\end{align*}
\begin{prop}.
    We have an identification of virtual  $\TT$-representations
    \begin{align*}        T^{\vir}_{\overline{\pi}}= \overline{K}Q_{\overline{\pi}}+K\overline{Q_{\overline{\pi}}}-P_{1234}Q_{\overline{\pi}}\overline{Q_{\overline{\pi}}}-\sum_{i=1}^4K_it_i \overline{Q_{\overline{\pi}}}-\sum_{i=1}^4\overline{K_i}t^{-1}_i Q_{\overline{\pi}}.
    \end{align*}
\end{prop}
\begin{proof}
    The virtual tangent space at a $\TT$-fixed point corresponding to a tuple of plane partitions $\overline{\pi}$ is computed as a class in $K$-theory by
    \begin{align}\label{eqn: virtual complex}
T^{\vir}_{\overline{\pi}}=T_{\CM^{\mathrm{nc}}_{\overline{r}, n} }|_{\overline{\pi}}-\CL|_{\overline{\pi}}+T^*_{\CM^{\mathrm{nc}}_{\overline{r}, n} }|_{\overline{\pi}}\in K_\TT^0(\pt).
    \end{align}
    The tangent space $T_{\CM^{\mathrm{nc}}_{\overline{r}, n} }|_{\overline{\pi}} $ at the fixed quotient corresponding to  $\overline{\pi}\in \CM^{\mathrm{nc}}_{\overline{r}, n} $ sits in an exact sequence of $\TT$-representations
    \begin{align*}
        0\to \Hom(Q_{\overline{\pi}}, Q_{\overline{\pi}})\to \BC^4\otimes \Hom(Q_{\overline{\pi}},Q_{\overline{\pi}})\oplus \bigoplus_{i=1}^4\bigoplus_{j=1}^{r_i}\Hom(\BC \cdot w_{ij}, Q_{\overline{\pi}})\to T_{\CM^{\mathrm{nc}}_{\overline{r}, n} }|_{\overline{\pi}}\to 0,
    \end{align*}
    where the middle representation  is the tangent space of the affine space $\mathsf R_{\overline r,n} $ and the first representation records the free action of $\GL(\BC^n)$. The weight decomposition of $\BC^4$ is
    \[\BC^4=t_1^{-1}+t_2^{-1}+t_3^{-1}+t_4^{-1},\]
   by which one readily computes in $K$-theory
    \[
    T_{\CM^{\mathrm{nc}}_{\overline{r}, n} }|_{\overline{\pi}}=(t_1^{-1}+t_2^{-1}+t_3^{-1}+t_4^{-1}-1)Q_{\overline{\pi}}\overline{Q_{\overline{\pi}}}+ \overline{K}Q_{\overline{\pi}}.
    \]
   The vector bundle $\CL$ at the point $\overline{\pi}$ satisfies
   \[
\CL|_{\overline{\pi}}=\Lambda^2\BC^4\otimes \Hom(Q_{\overline{\pi}}, Q_{\overline{\pi}})\oplus\bigoplus_{i=1}^4\Hom(K_i\cdot t_i, Q_{\overline{\pi}})\oplus \bigoplus_{i=1}^4\Hom(K_i\cdot t_i, Q_{\overline{\pi}})^*,
   \]
   where the weight decomposition of $ \Lambda^2\BC^4$ is 
   \[
   \Lambda^2\BC^4=t_1^{-1}t_2^{-1}+ t_1^{-1}t_3^{-1}+t_1^{-1}t_4^{-1}+t_2^{-1}t_3^{-1}+t_2^{-1}t_4^{-1}+t_3^{-1}t_4^{-1},
   \]
   which implies that 
   \[
   \CL|_{\overline{\pi}}=(t_1^{-1}t_2^{-1}+ t_1^{-1}t_3^{-1}+t_1^{-1}t_4^{-1}+t_2^{-1}t_3^{-1}+t_2^{-1}t_4^{-1}+t_3^{-1}t_4^{-1})Q_{\overline{\pi}}\overline{Q_{\overline{\pi}}}+\sum_{i=1}^4K_it_i \overline{Q_{\overline{\pi}}}+\sum_{i=1}^4\overline{K_i}t^{-1}_i Q_{\overline{\pi}}.
   \]
   Plugging the computations in \eqref{eqn: virtual complex} and exploiting the relation $t_1t_2t_3t_4=1$ concludes the proof.
\end{proof}
We exhibit a square root of the virtual tangent bundle at each fixed point. Let $\overline{\pi}$ be a tuple of plane partitions. Set the \emph{vertex term} 
\begin{align}\label{eqn: vertex term}
    \mathsf{v}_{\overline{\pi}}=\overline{K}Q_{\overline{\pi}}-\sum_{j=1}^4K_j t_j \overline{Q_{\overline{\pi}}}-\sum_{ j=1}^4 \overline{P_{j_1j_2j_3}}Q_{\overline{\pi}_j}\overline{Q_{\overline{\pi}_j}}-\sum_{1\leq i<j\leq 4} \overline{P_{j_1j_2j_3}}\left(Q_{\overline{\pi}_j}\overline{Q_{\overline{\pi}_i}}+Q_{\overline{\pi}_i}\overline{Q_{\overline{\pi}_j}}\right),
\end{align}
 where   $\{j_1, j_2, j_3\}$ are the three indices in $\{1,2,3,4\}$ different from $j$. It is easily seen to satisfy
\begin{align*}
T^{\vir}_{\overline{\pi}}=\mathsf{v}_{\overline{\pi}}+\overline{\mathsf{v}_{\overline{\pi}}},
\end{align*}
thanks to the identity
\[
P_{j_1j_2j_3}+\overline{P_{j_1j_2j_3}}=P_{1234},
\]
which exploits the Calabi-Yau condition $t_1t_2t_3t_4=1$. 
Thus, by \eqref{eqn: square roots with brackets 2} we can compute the partition function of tetrahedron instantons as
\begin{align*}
 \CZ_{\overline{r}}(q)=\sum_{\overline{\pi}}(-1)^{\sigma_{\overline{\pi}}}[-\mathsf{v}_{\overline{\pi}}]\cdot q^{|\overline{\pi}|},
\end{align*}
where the sign $(-1)^{\sigma_{\overline{\pi}}}$ receives a contribution both from $ (-1)^{{o_{\overline{\pi}}}}$ (coming from the orientation of the fixed locus) and from the choice of the square root of the virtual tangent bundle.

We postpone to Appendix \ref{sec: app} the proof of the correct sign rule associated with our choice of square root (cf. Corollary \ref{cor: final sign}).
\begin{theorem}\label{thm: correct sign body text}
    We have that
    \[
    (-1)^{\sigma_{\overline{\pi}}}=1.
    \]
\end{theorem}
Thanks to Theorem \ref{thm: correct sign body text}, we can explicitly express the partition function as
\begin{align}\label{eqn: partition function with vertex terms}
 \CZ_{\overline{r}}(q)=\sum_{\overline{\pi}}[-\mathsf{v}_{\overline{\pi}}]\cdot q^{|\overline{\pi}|}.
\end{align}
where the sum runs over tuples $ \overline{\pi}=(\overline{\pi}_1, \dots, \overline{\pi}_4)$, where each $\overline{\pi}_i$ is an $r_i$-tuple of plane partitions.
\subsection{Factorisation}

Let $\overline{r}=(r_1, \dots, r_4)$ and $ \overline{\pi}=(\overline{\pi}_1, \dots, \overline{\pi}_4)$, where each $\overline{\pi}_i=(\pi_{il})_{il}$ is an $r_i$-tuple of plane partitions, for $i=1,\dots, 4$, $l=1, \dots, r_i$. The vertex term can be written as a sum
\begin{align*}
\mathsf{v}_{\overline{\pi}}&=\sum_{1\leq i\leq j\leq 4}\mathsf{v}^{(ij)}_{\overline{\pi}},\\
\mathsf{v}^{(ij)}_{\overline{\pi}}&=\begin{cases}
\overline{K}_j Q_{\overline{\pi}_j}-K_jt_j\overline{ Q_{\overline{\pi}_j}}-\overline{P_{j_1j_2j_3}}Q_{\overline{\pi}_j}\overline{ Q_{\overline{\pi}_j}} & i=j,\\
 \overline{K}_i Q_{\overline{\pi}_j}-K_jt_j\overline{ Q_{\overline{\pi}_i}}+ \overline{K}_j Q_{\overline{\pi}_i}-K_it_i\overline{ Q_{\overline{\pi}_j}}-\overline{P_{j_1j_2j_3}}(Q_{\overline{\pi}_j}\overline{ Q_{\overline{\pi}_i}}+Q_{\overline{\pi}_i}\overline{ Q_{\overline{\pi}_j}})& i<j.
\end{cases}
\end{align*}
If $i=j$, the vertex term $\mathsf{v}^{(ii)}_{\overline{\pi}}$ recovers  the vertex term of the higher rank Donaldson-Thomas theory of $\BC^3_i$, for the tuple of plane partitions $\overline{\pi}_i$ (cf.~\cite[Sec.~5.1]{FMR_higher_rank}). For this identification, we crucially use that $t_i=(t_{i_1}t_{i_2}t_{i_3})^{-1}$ where  $\{i_1, i_2, i_3\}$ are the three indices in $\{1,2,3,4\}$ different from $i$. To ease the notation, we define  the \emph{Calabi-Yau weight} of the coordinate hyperplane $\BC^3_i$ to be
\[\kappa_i:=t_{i_1}t_{i_2}t_{i_3}=t_i^{-1}.\] 
By \eqref{eqn: decomposition of Q} we can further identify the terms $ \mathsf{v}^{(ij)}_{\overline{\pi}}$ as sums
\begin{align*}
\mathsf{v}^{(ij)}_{\overline{\pi}}&=\sum_{\substack{1\leq l\leq r_i\\ 1\leq k\leq r_j}}\mathsf{v}^{(ij, lk)}_{\overline{\pi}},\\
    \mathsf{v}^{(ij, lk)}_{\overline{\pi}}&=\begin{cases}
     w_{jl}^{-1}w_{jk}\left(Z_{\pi_{jk}}-\kappa_j^{-1}\overline{Z_{\pi_{jl}}}-\overline{P_{j_1j_2j_3}}Z_{\pi_{jk}}\overline{Z_{\pi_{jl}}}\right)& i=j\\
    w_{il}^{-1}w_{jk}\left(Z_{\pi_{jk}}-\kappa^{-1}_j\overline{ Z_{\pi_{il}}}-\overline{P_{j_1j_2j_3}}Z_{\pi_{jk}}\overline{ Z_{\pi_{il}}} \right)+  w_{il}w_{jk}^{-1}\left(Z_{\pi_{il}}-\kappa^{-1}_i\overline{ Z_{\pi_{jk}}}-\overline{P_{j_1j_2j_3}}(Z_{\pi_{il}}\overline{ Z_{\pi_{jk}}})\right)& i< j,
        \end{cases}
\end{align*}
where each $Z_{il}$ are the global sections of a zero-dimensional closed subscheme of $\BC^3_i$, whose monomial ideal corresponds to the plane partition $\pi_{il}$. If $i=j$ and $l=k$, the vertex term $\mathsf{v}^{(ii, ll)}_{\overline{\pi}}$ recovers  the vertex term of the rank one Donaldson-Thomas theory of $\BC^3_i$, for the  plane  partition $\pi_{il}$ (cf.~\cite[Eqn. (12)]{MNOP_1}).

As already anticipated, the vertex term is $\TT$-movable, which in particular implies that the rational function $[-\mathsf{v}_{\overline{\pi}}]$ is well-defined and non-zero.
\begin{prop}\label{prop: t movable}
    The vertex term $\mathsf{v}_{\overline{\pi}}$ is $\TT$-movable.
\end{prop}
\begin{proof}
By the decomposition of the vertex term as
\begin{align}\label{eqn: decomposition vertex}
   \mathsf{v}_{\overline{\pi}}=\sum_{1\leq i\leq j\leq 4}\sum_{\substack{1\leq l\leq r_i\\ 1\leq k\leq r_j}}\mathsf{v}^{(ij, lk)}_{\overline{\pi}},
\end{align}
it is clear that the possible contributions to the $\TT$-fixed part have to come from the    \emph{diagonal} terms $\mathsf{v}^{(ii, ll)}_{\overline{\pi}}$, which are $\TT$-movable by \cite[Lemma 8]{MNOP_1} (see also \cite[Prop. 3.14]{FMR_higher_rank}).
\end{proof}
\subsubsection{Rigidity}
%The partition function of tetrahedron instantons is a priori a rational function depending on the variables $t_i, w_{ij}$ for $i=1,\dots, 4$ and $j=1,\dots, r_i$.
Analogously to the arguments of \cite[Sec.~3.1]{Arb_K-theo_surface}, we establish a general criterion which determines the possible poles that can appear in the denominator after  $K$-theoretic localisation. Let $\CM$ be a quasi-projective scheme, acted by an    algebraic torus  $\BT$, such that the fixed locus $\CM^\BT$ is proper.
%Let $\CM$ be a quasi-projective scheme, endowed with a 3-term  \emph{symmetric obstruction theory} $\BE\to \BL_\CM$  satisfying the \emph{isotropic condition} in the sense of Park \cite[Def. 1.9, 1.10]{Park_DT_pullbacks}. Then $\CM$ admits a virtual structure sheaf $\widehat{\oO}
%_\CM^{\vir}\in K_0(\CM, \BZ\left[\tfrac{1}{2}\right])$. Examples include the toy model situation of Section \ref{sec: virtual classes}, and the case of $ \CM$ admitting a derived enhancement to a $-2$-shifted symplectic scheme  \cite{PTVV_shifted_symplectic} and an orientation \cite{CGJ_orientability} (cf. \cite[Ex. 1.17]{Park_DT_pullbacks},\cite[Def. 5.9]{OT_1}). A  concrete example of the latter class is $\CM$ a fine moduli space of  stable sheaves on a Calabi-Yau fourfold.  Assume that $\CM$ is acted by an    algebraic torus  $\BT$, such that $\BE$ admits a lift to a $\BT$-equivariant symmetric obstruction theory. 

\begin{definition}[{\cite[Def. 3.1]{Arb_K-theo_surface}}]
  Given a weight $w\in \widehat\BT$, denote by $\BT_w\subset \BT$  the maximal torus inside the subgroup $\ker(w)\subset \BT$. A weight $w\in \widehat{\BT}$ is called  a \emph{compact} weight of $\CM$ if $\CM^{\BT_w}$ is proper, and \emph{non-compact} otherwise.
\end{definition}
Let  $\CF\in K^\BT_0(\CM)$ be a $\BT$-equivariant class in $K$-theory.
As in \eqref{def: inv by loc def}, by Thomason  $K$-theoretic localisation \cite{Tho:formule_Lefschetz}, we have that 
\begin{align*}
    \chi(\CM, \CF)\in \BQ[t_1^{\pm 1}, \dots, t_m^{\pm 1}]\left[\frac{1}{1-w}:w\in \widehat\BT\right],
\end{align*}
where $m=\rk \BT$ is the rank of the torus. In other words, the pushforward $K_0^\BT(\CM^\BT)\to K_0^\BT(\CM) $ becomes an isomorphism after inverting  virtual characters  of the form $1-t^{\mu}$.  We generalize \cite[Prop. 3.2]{Arb_K-theo_surface} -- spelled out in loc. cit. for virtual structures à la Behrend-Fantechi \cite{BF_normal_cone}.
\begin{prop}\label{prop: Noah}
    We have that 
    \[
    \chi(\CM, \CF)=\frac{p(t_1, \dots, t_m)}{\prod_w(1-w)},
    \]
    where $p(t)$ is a Laurent polynomial and  the product is over some non-compact weights $w$ of $\CM$ (with possible multiplicities).
\end{prop}
\begin{proof}
By definition of $K$-theoretic invariants \eqref{def: inv by loc def}, we have that 
\[ \chi^\BT(\CM,  \CF)=\frac{p(t_1,\dots, t_m)}{q(t_1, \dots, t_m)},\]
where $p(t), q(t)$ are Laurent polynomial and the superscript $\BT$ records that we defined the invariants $\BT$-equivariantly. Let $w$ be a compact weight. Up to replacing $\BT$ with a finite cover, we can assume that $w=(w_0)^e$, where 
$w_0=t_1^{\mu_1}\cdots t_m^{\mu_m}$ is weight such that at least one $|\mu_i|= 1$. Then the subtorus $\BT_w$ coincides with $\ker(w_0)$.
%We say that a non-trivial weight $w\in \widehat{\BT}$ is \emph{primitive} if it is not the power of another   weight, i.e. if $w\neq \tilde{w}^{n}$ for all $n\geq 2$ and $\tilde{w}\in \widehat{\BT}$.
Since the  $\BT_w$-fixed locus $\CM^{\BT_w}$ is proper, we can define the invariants $\BT_w$-equivariantly, and get
\[ \chi^{\BT_w}(\CM,  \CF)=\frac{\tilde{p}(t_1,\dots, t_m)}{\tilde{q}(t_1, \dots, t_m)},\]
where $\tilde{p}(t), \tilde{q}(t)$ are  the Laurent polynomials coming by  the restricting $p(t), q(t)$ as $\BT$-representations to $\BT_w$-representations. But the invariants
\[\chi^{\BT_w}(\CM, \CF)\in \frac{\BQ[t_1^{\pm 1}, \dots, t_m^{\pm 1}]}{(1-w_0)}\left[\frac{1}{1-v}:v\in \widehat\BT, v\neq w_0^n\right]\]
are well-defined, which  implies that $(1-w_0^n)$ cannot appear as factors of $q(t)$, for all $n\in \BZ$.
\end{proof}
\subsubsection{Framing limits}
We prove in this section Theorem \ref{thm: intro main thm} from the Introduction.  We first show that the partition function of tetrahedron instantons does not depend on the framing parameters $w_{ij}$,  exploiting a suitable rigidity argument. Granting this independence, we  specialise the framing parameters $w_{ij}$ to arbitrary values and send them to infinity, which recovers the formula in Theorem \ref{thm: intro main thm}.

\begin{theorem}\label{thm: framinh independence}
    The partition function of tetrahedron instantons $\CZ_{\overline{r}}(q)$ does not depend on the weights $w_{il}$, for $i=1, \dots, 4$ and $l=1, \dots, r_i$.
\end{theorem}
\begin{proof}
    We prove this  result  by exploiting similar arguments to \cite[Thm. 6.5]{FMR_higher_rank}.    The $n$-th coefficient of $\CZ_{\overline{r}}(q)$ is a sum of contributions
\[
[-\mathsf{v}_{\overline{\pi}}],\qquad \lvert \overline{\pi}\rvert = n.
\]
A simple manipulation shows that
\begin{equation}\label{eqn:poles_showing_up}
[-\mathsf{v}_{\overline{\pi}}] = A(t)\prod_{(i,l)\neq (j,k)} \frac{\prod_{\mu_{ij,lk}}(1-w_{il}^{-1}w_{jk} t^{\mu_{ij,lk}})}{\prod_{\nu_{ij,lk}}(1-w_{il}^{-1}w_{jk} t^{\nu_{ij,lk}})},
\end{equation}
where the product is over all $(i,l)\neq (j,k) $ such  that $i,j=1, \dots, 4$ and $l=1, \dots, r_i$ and $j=1, \dots, r_j$ while  $A(t)\in \BQ(\!(t_1^{\frac{1}{2}},t_2^{\frac{1}{2}},t_3^{\frac{1}{2}}, t_4^{\frac{1}{2}})\!)/(t_1t_2t_3t_4-1) $ and the number of weights $\mu_{ij,lk}$ and $\nu_{ij,lk}$ is the same. Thus, $\CZ_{\overline{r}}(q)$ is a homogeneous rational expression of total degree 0 with respect to the variables $w_{il}$. We aim to show that $\CZ_{\overline{r}}(q)$ has  no poles of the form $1-w_{il}^{-1}w_{jk} t^{\nu}$, implying that it is a degree $0$ polynomial in the $w_{il}$, hence constant in the $w_{il}$.

Set $\mathsf w =  w_{il}^{-1}w_{jk}t^\nu$ for fixed $(i,l)\neq (j,k)$ and $\nu \in \widehat{\TT}_0$. To see that $1-\mathsf w$ is not a pole, we show that $ \mathsf w$ is a non-compact weight and use Proposition \ref{prop: Noah}. In other words,   we need to prove that the fixed locus of the torus  $\TT_{\mathsf w} = \ker (\mathsf w)$ is proper.

 Consider, consider the automorphism $\tau_\nu\colon \TT \simto \TT$ defined by
\[
(t_1,t_2,t_3, t_4,w_{11}, \dots, w_{4r_4}) \mapsto (t_1,t_2,t_3, t_4,w_{11},\ldots,w_{il}t^{-\nu},\ldots,w_{jk},\ldots,w_{4r_4}).
\]
It maps $\TT_{\mathsf w}\subset \TT$ isomorphically onto the subtorus $\TT_0 \times \set{w_{il}=w_{jk}} \subset \TT$. This yields an inclusion of tori
\begin{equation}\label{inclusions_tori}
\TT_0 \simto \TT_0 \times \Set{(1,\ldots,1)} \into \tau(\TT_{\mathsf w}).
\end{equation}
We consider the action $\sigma_\nu \colon \TT \times \CM_{\overline{r}, n} \to \CM_{\overline{r}, n} $ where $\TT_0$ translates the support of the quotient sheaf in the usual way, the $l$-th summand of $\iota_{i,*}\OO^{\oplus r_i}$ gets scaled by $w_{il}t^\nu$ and all other summands by $w_{jk}$ for $(j,k)\neq (i,l)$. 
%In other words, in terms of the quiver description of $\CM_{\overline{r}, n}$, we set
%\[
%\igma_\nu(\mathbf t,(A_1,A_2,A_3,u_1,\ldots,u_r)) = (t_1A_1,t_2A_2,t_3A_3,w_1u_1,\ldots,w_it^\nu u_i,\ldots,w_ru_r),
%\]
%just a variation of Equation \eqref{T-action_on_NCQUOT} in the $i$-th vector component. 
 We have a commutative diagram
\[
\begin{tikzcd}[row sep=large]
\TT_{\mathsf w} \times \CM_{\overline{r}, n} \arrow{r}{\sigma}\arrow[swap]{d}{\tau_\nu\times\id} & \CM_{\overline{r}, n} \arrow[equal]{d} \\
\tau_\nu(\TT_{\mathsf w}) \times \CM_{\overline{r}, n} \arrow{r}{\sigma_\nu} & \CM_{\overline{r}, n} 
\end{tikzcd}
\]
where $\sigma$ is the restriction of the usual $\TT$-action on $\CM_{\overline{r}, n}$. This diagram induces a natural isomorphism $\CM_{\overline{r}, n} ^{\TT_{\mathsf w}}\simto\CM_{\overline{r}, n} ^{\tau_\nu(\TT_{\mathsf w})}$, which combined with \eqref{inclusions_tori}
yields an inclusion
\[
\CM_{\overline{r}, n} ^{\TT_{\mathsf w}}\simto\CM_{\overline{r}, n} ^{\tau_\nu(\TT_{\mathsf w})} \into \CM_{\overline{r}, n} ^{\TT_0},
\]
where $\CM_{\overline{r}, n} ^{\TT_0}$ is the fixed locus with respect to the action $\sigma_\nu$. But by the same reasoning as in Remark \ref{lemma:compact_fixed_locus_Quot}, this fixed locus is proper (because, again, a $\TT_0$-fixed surjection $\CE_{\overline{r}, n}\onto T$ necessarily has the quotient $T$ entirely supported at the origin $0 \in \Delta$). Thus $\mathsf w$ is a compact weight, and the result follows.
\end{proof}
By Theorem \ref{thm: framinh independence}  the partition function of tetrahedron instantons does not depend on the framing parameters $w_{il}$, therefore we can compute it by specialising them to arbitrary values. We set $w_{il}=L^{N_{il}}$, where $N_{il}\gg 0$ are large integers satisfying  $N_{il}>N_{ik}$ for $k>l$ and $ N_{jk}\gg N_{il}$ for $j>i$, and take limits $L\to \infty$.
We  totally order the indices $(i,j)$ lexicographically, setting $(j,k)>(i,l)$ if $j>i$ or $i=j$ and $k>l$.
\begin{prop}\label{prop: limits}
    Let $(j,k)>(i,l)$. If $i=j$, we have
    \begin{align*}
        \lim_{L\to \infty}[-\mathsf{v}^{(ij, lk)}_{\overline{\pi}}][-\mathsf{v}^{(ji, kl)}_{\overline{\pi}}]|_{w_{il}=L^{N_{il}}}=(-\kappa_{i}^{\frac{1}{2}})^{|\pi_{jk}|}(-\kappa_j^{-\frac{1}{2}})^{|\pi_{il}|}.
    \end{align*}
    If $i<j$, we have
    \[ \lim_{L\to \infty}[-\mathsf{v}^{(ij, lk)}_{\overline{\pi}}]|_{w_{il}=L^{N_{il}}}=(-\kappa_{i}^{\frac{1}{2}})^{|\pi_{jk}|}(-\kappa_j^{-\frac{1}{2}})^{|\pi_{il}|}.
    \]
\end{prop}
\begin{proof}
    For all monomials $t^{\mu}$, we have the limits
    \begin{align*}
         \lim_{L\to \infty}[w_{il}^{-1}w_{jk}t^{\mu}]|_{w_{il}=L^{N_{il}}}&=t^{\frac{\mu}{2}}\cdot  \lim_{L\to \infty}L^{\frac{N_{jk}-N_{il}}{2}},\\
          \lim_{L\to \infty}[w_{jk}^{-1}w_{il}t^{\mu}]|_{w_{il}=L^{N_{il}}}&=(-t^{-\frac{\mu}{2}})\cdot  \lim_{L\to \infty}L^{\frac{N_{jk}-N_{il}}{2}}.
    \end{align*}
    Set $Z_{\pi_{jk}}=\sum_{\nu}t^\nu$ and $Z_{\pi_{il}}=\sum_{\mu}t^\mu$. Taking limits, we have
    \begin{align*}
         \lim_{L\to \infty}[-\mathsf{v}^{(ij, lk)}_{\overline{\pi}}]|_{w_{il}=L^{N_{il}}}&=\lim_{L\to \infty}[w_{il}^{-1}w_{jk}(-Z_{\pi_{jk}}+\kappa_j^{-1}\overline{Z_{\pi_{il}}}+\overline{P_{j_1j_2j_3}}Z_{\pi_{jk}}\overline{Z_{\pi_{il}}})]|_{w_{il}=L^{N_{il}}}\\
         &=\lim_{L\to \infty}L^{\frac{N_{jk}-N{il}}{2}(|\pi_{jk}|-|\pi_{il}|)}\cdot \frac{\prod_{\mu}t^{-\frac{\mu}{2}}}{   \prod_{\nu}t^{\frac{\nu}{2}}}   \cdot \kappa_j^{-\frac{|\pi_{il}|}{2}}.
    \end{align*}
    A similar computation yields
    \begin{align*}
        \lim_{L\to \infty}[-\mathsf{v}^{(ji, kl)}_{\overline{\pi}}]|_{w_{il}=L^{N_{il}}}= (-1)^{|\pi_{jk}|-|\pi_{il}|}\lim_{L\to \infty}L^{\frac{N_{jk}-N_{il}}{2}}\cdot  \frac{\prod_{\nu}t^{\frac{\nu}{2}}}{   \prod_{\mu}t^{-\frac{\mu}{2}}} \cdot \kappa_i^{\frac{|\pi_{jk}|}{2}},
    \end{align*}
    by which we conclude that 
    \begin{align*}
         \lim_{L\to \infty}[-\mathsf{v}^{(ij, lk)}_{\overline{\pi}}][-\mathsf{v}^{(ji, kl)}_{\overline{\pi}}]|_{w_{il}=L^{N_{il}}}=(-\kappa_{i}^{\frac{1}{2}})^{|\pi_{jk}|}(-\kappa_j^{-\frac{1}{2}})^{|\pi_{il}|}.
    \end{align*}
    The second claim follows by an analogous computation.
\end{proof}
For an integer $m\in \BZ$, we set  the \emph{sign function} 
\[
\sgn(m)=\begin{cases}
1 & m>0\\
-1 & m<0\\
0 & m=0.
\end{cases}
\]
\begin{lemma}\label{lemma combin}
Let $x_i$ be variables and $c_i\in \BZ$ be integers, for $i=1, \dots, r$.  Then we have
\[
\prod_{1\leq i <j\leq r} x_i^{c_j}x_j^{-c_i}=\prod_{i=1}^r\prod_{j=1}^r x_j^{\sgn(i-j)c_i}.
\]
\end{lemma}
\begin{proof}
    For $r=1$ the statement holds  trivially. By induction, suppose that it holds for $r-1$. Then we have
    \begin{align*}
        \prod_{1\leq i <j\leq r} x_i^{c_j}x_{j}^{-c_i}&= \prod_{1\leq i <j\leq r-1}x_i^{c_j}x_j^{-c_i}\cdot \prod_{i=1}^{r-1}x_i^{c_r}x_r^{-c_i}\\
        &=\prod_{i,j=1}^{r-1}  x_j^{\sgn(i-j)c_i}\cdot \prod_{i=1}^{r-1}x_i^{c_r}x_r^{-c_i}\\
        &=\prod_{i,j=1}^r x_j^{\sgn(i-j)c_i}.
    \end{align*}
\end{proof}
If $\overline{r}=(0,0,0,1)$ the partition function of tetrahedron instantons recovers\footnote{A direct way to see this is that if $\overline{r}=(0,0,0,1)$ then the vertex term \eqref{eqn: vertex term} coincides with the one used by Okounkov \cite[Eqn. (3.4.30)]{Okounk_Lectures_K_theory}.} the  partition function of \emph{K}-theoretic Donaldson-Thomas theory of $\BC^3$, whose explicit formula had been proved by Okounkov\ \cite[Thm. 3.3.6]{Okounk_Lectures_K_theory}
\begin{align}\label{eqn: Oku partition function}
    \CZ_{(0,0,0,1)}(-q)=\Exp\left(\frac{[t_1t_2][t_1t_3][t_2t_3]}{[t_1][t_2][t_3]}\frac{1}{[\kappa_4^{\frac{1}{2}} q][\kappa_4^{\frac{1}{2}}q^{-1}]}, \right),
\end{align}
where the plethystic exponential is defined as follows. For any formal power series $f(p_1, \ldots, p_r; q_1, \ldots, q_s)$ in $\BQ(p_1, \ldots, p_r)[\![q_1, \ldots, q_s]\!]$, such that $f(p_1, \ldots, p_r;0,\ldots,0)=0$, its \emph{plethystic exponential} is defined as 
\begin{align}\label{eqn: on ple} 
\Exp(f(p_1, \ldots, p_r;q_1, \ldots, q_s)) &:= \exp\Big( \sum_{n=1}^{\infty} \frac{1}{n} f(p_1^n, \ldots, p_r^n;q_1^n, \ldots, q_s^n) \Big),
\end{align}
viewed as an element of $\BQ(p_1, \ldots, p_r)[\![q_1, \ldots, q_s]\!]$.

Notice that $\CZ_{(0,0,0,1)}(q)$ is symmetric in the variables $t_1, t_2, t_3$, and invariant under $q\to q^{-1}$. We define the partition function
\begin{align*}
    \CZ^{(i)}(q)=\CZ_{(0,0,0,1)}(q)|_{t_1=t_{i_1}, t_2=t_{i_2}, t_3=t_{i_3}}, \quad i=1, \dots, 4,
\end{align*}
which  is just the partition function \eqref{eqn: Oku partition function} in the three variables $\{i_1, i_2, i_3\}$ different than $i$. In other words, $ \CZ^{(i)}(q)$ is the tetrahedron partition function for the tuple  $e_i$, which is $4$-tuple with  $1$ in the $i$-th entry and $0$ in all the others zero.
\begin{theorem}\label{thm:factorization}
    Let $\overline{r}=(r_1, \dots, r_4)$ and $r=\sum_{i=1}^4r_i$. We have the factorisation 
    \begin{align*}
        \CZ_{\overline{r}}(q)=\prod_{i=1}^4 \prod_{l=1}^{r_i}\CZ^{(i)}\left((-1)^{r+1}q\kappa_i^{\frac{-r_i-1}{2}+l}\prod_{j=1}^4\kappa_j^{\frac{r_j\cdot\sgn(i-j)}{2}}\right).
    \end{align*}
\end{theorem}
\begin{proof}
   Set $w_{il}=L^{N_{il}}$. By Theorem \ref{thm: framinh independence} the partition function $\CZ_{\overline{r}}(q)$ can be computed in the limit $L\to \infty$. By  Proposition \ref{prop: limits}, Lemma \ref{lemma combin} and \eqref{eqn: partition function with vertex terms} we have
\begin{align*}
 \CZ_{\overline{r}}(q) &= \lim_{L\to \infty}\sum_{\overline{\pi}} q^{|\overline{\pi}|}\prod_{1\leq i\leq j\leq 4} \prod_{\substack{1\leq l\leq r_i\\ 1\leq k\leq r_j}} [-\mathsf{v}^{(ij, lk)}_{\overline{\pi}}]\\
     &= \lim_{L\to \infty}\sum_{\overline{\pi}}\prod_{i=1}^4 \prod_{l=1}^{r_i} q^{|\pi_{il}|}[-\mathsf{v}_{\overline{\pi}}^{(ii, ll)}]\prod_{i=1}^4\prod_{1\leq k<l\leq r_i}[-\mathsf{v}^{(ii, lk)}_{\overline{\pi}}][-\mathsf{v}^{(ii, kl)}_{\overline{\pi}}]\prod_{1\leq i<j\leq 4}\prod_{\substack{1\leq l\leq r_i\\ 1\leq k\leq r_j}}[-\mathsf{v}^{(ij, lk)}_{\overline{\pi}}]\\
    &= \sum_{\overline{\pi}}\prod_{i=1}^4 \prod_{l=1}^{r_i} q^{|\pi_{il}|}[-\mathsf{v}_{\overline{\pi}}^{(ii, ll)}]\prod_{(i,l)<(j,k)}(-\kappa_{i}^{\frac{1}{2}})^{|\pi_{jk}|}(-\kappa_j^{-\frac{1}{2}})^{|\pi_{il}|}\\
    &= \sum_{\overline{\pi}}\prod_{i=1}^4 \prod_{l=1}^{r_i} q^{|\pi_{il}|}[-\mathsf{v}_{\overline{\pi}}^{(ii, ll)}]\prod_{i=1}^4\prod_{l=1}^{r_i}\left(\prod_{j=1}^4\prod_{k=1}^{r_j}(-\kappa_j^{\frac{1}{2}})^{\sgn((i,l)-(j,k))}\right)^{|\pi_{il}|},
    \end{align*}
where the sign function is defined as
\[
\sgn((i,l)-(j,k))=\begin{cases}
1 & (i,l)<(j,k)\\
-1 & (i,l)>(j,k)\\
0 & (i,l)=(j,k).
\end{cases}
\]
For a fixed $(i,l)$, counting the number of $(j,k)$ such that $(i,l)<(j,k)$ with respect to the lexicographic order yields the identity
\[
\prod_{j=1}^4\prod_{k=1}^{r_j}(-\kappa_j^{\frac{1}{2}})^{\sgn((i,l)-(j,k))}=(-1)^{r+1}\kappa_i^{\frac{-r_i-1}{2}+l}\prod_{j=1}^4\kappa_j^{\frac{r_j\cdot\sgn(i-j)}{2}}.
\]
Therefore we have
\begin{align*}
     \CZ_{\overline{r}}(q) &= \sum_{\overline{\pi}}\prod_{i=1}^4 \prod_{l=1}^{r_i}[-\mathsf{v}_{\overline{\pi}}^{(ii, ll)}]\left((-1)^{r+1}q\kappa_i^{\frac{-r_i-1}{2}+l}\prod_{j=1}^4\kappa_j^{\frac{r_j\cdot\sgn(i-j)}{2}} \right)^{|\pi_{il}|}\\
     &=\prod_{i=1}^4 \prod_{l=1}^{r_i}\sum_{\pi_{il}}[-\mathsf{v}_{\overline{\pi}}^{(ii, ll)}]\left((-1)^{r+1}q\kappa_i^{\frac{-r_i-1}{2}+l}\prod_{j=1}^4\kappa_j^{\frac{r_j\cdot\sgn(i-j)}{2}} \right)^{|\pi_{il}|}\\
     &=\prod_{i=1}^4 \prod_{l=1}^{r_i}\CZ^{(i)}\left((-1)^{r+1}q\kappa_i^{\frac{-r_i-1}{2}+l}\prod_{j=1}^4\kappa_j^{\frac{r_j\cdot\sgn(i-j)}{2}}\right).
\end{align*}
\end{proof}
Combining Theorem \ref{thm:factorization} with the explicit expression \eqref{eqn: Oku partition function}, we prove a closed formula for the partition function of tetrahedron instantons. We set 
\[\kappa_{\overline{r}}:=\prod_{i=1}^4 \kappa_i^{r_i}.\]
\begin{theorem}\label{thm: explicit expression inv}
    Let $\overline{r}=(r_1, \dots, r_4)$  and $r=\sum_{i=1}^4r_i$. We have 
    \[
    \CZ_{\overline{r}}((-1)^rq)=\Exp\left(-\frac{[t_1t_2][t_1t_3][t_2t_3]}{[t_1][t_2][t_3][t_4]}\frac{[\kappa_{\overline{r}}]}{[\kappa_{\overline{r}}^{\frac{1}{2}} q][\kappa_{\overline{r}}^{\frac{1}{2}}q^{-1}]} \right).
    \]
\end{theorem}
\begin{proof}
    By Theorem \ref{thm:factorization} and taking plethystic logarithm, i.e. the inverse of the plethystic exponential $\Exp$, we just need show that 
    \begin{align}\label{eqn: first eqn}
           \sum_{i=1}^4\sum_{l=1}^{r_i}\frac{[t_{i_1}t_{i_2}][t_{i_1}t_{i_3}][t_{i_2}t_{i_3}]}{[t_{i_1}][t_{i_2}][t_{i_3}]}\frac{1}{[\kappa_i^{\frac{1}{2}} q_{il}][\kappa_i^{\frac{1}{2}}q^{-1}_{il}]}=-\frac{[t_1t_2][t_1t_3][t_2t_3]}{[t_1][t_2][t_3][t_4]}\frac{[\kappa_{\overline{r}}]}{[\kappa_{\overline{r}}^{\frac{1}{2}} q][\kappa_{\overline{r}}^{\frac{1}{2}}q^{-1}]},
    \end{align}
    where  $\{i_1, i_2, i_3\}$ are the three indices in $\{1,2,3,4\}$ different from $i$ and   we define 
    \[
    q_{il}=q\kappa_i^{\frac{-r_i-1}{2}+l}\prod_{j=1}^4\kappa_j^{\frac{r_j\cdot\sgn(i-j)}{2}}.
    \]
    By the relations $t_1t_2t_3t_4=1$ and $\kappa_i=t_i^{-1}$, the claim \eqref{eqn: first eqn} is equivalent to
    \begin{align}\label{eqn: second eqn}
         \sum_{i=1}^4\sum_{l=1}^{r_i}\frac{[\kappa_i]}{[\kappa_i^{\frac{1}{2}} q_{il}][\kappa_i^{\frac{1}{2}}q^{-1}_{il}]}=\frac{[\kappa_{\overline{r}}]}{[\kappa_{\overline{r}}^{\frac{1}{2}} q][\kappa_{\overline{r}}^{\frac{1}{2}}q^{-1}]}.
    \end{align}
    We prove a slightly more general statement: let $x_1,\dots, x_r$ be variables, and let $q_i=q\prod_{j=1}^rx_j^{\frac{\sgn(i-j)}{2}}$. Then we claim that 
    \begin{align}\label{eqn: third eqn}
     \sum_{i=1}^r\frac{[x_i]}{[x_i^{\frac{1}{2}} q_{i}][x_i^{\frac{1}{2}}q^{-1}_{i}]}=\frac{[\prod_{i=1}^rx_i]}{[\prod_{i=1}^rx_i^{\frac{1}{2}} q][\prod_{i=1}^rx_i^{\frac{1}{2}}q^{-1}]}.
    \end{align}
    Clearly, \eqref{eqn: second eqn} follows by \eqref{eqn: third eqn} by setting $x_1,\dots, x_{r_1}=\kappa_1, \dots, x_{r_1+r_2+r_3+1}=\dots = x_r=\kappa_4$. 
    
    We prove \eqref{eqn: third eqn} by induction on $r$. If $r=1$, the claim is trivial, and if $r=2$, it is easy to check it holds. So let $r\geq 3$ and assume it holds for $r-2$. For $a=1, \dots, r-2$ set
    \begin{align*}
        y_a&=x_{a+1},\\
        p&=qx_1^{\frac{1}{2}}x_r^{-\frac{1}{2}},\\
        p_a&=p\prod_{b=1}^{r-2}y_b^{\frac{\sgn(a-b)}{2}}.
    \end{align*}
    Then we have
    \begin{align*}
         \sum_{i=1}^r\frac{[x_i]}{[x_i^{\frac{1}{2}} q_{i}][x_i^{\frac{1}{2}}q^{-1}_{i}]}&= \sum_{a=1}^{r-2}\frac{[y_a]}{[y_a^{\frac{1}{2}} p_{a}][y_a^{\frac{1}{2}}p^{-1}_{a}]}+\frac{[x_1]}{[x_1^{\frac{1}{2}} q_{1}][x_1^{\frac{1}{2}}q^{-1}_{1}]}+\frac{[x_r]}{[x_r^{\frac{1}{2}} q_{r}][x_r^{\frac{1}{2}}q^{-1}_{r}]}\\
         &=\frac{[\prod_{a=1}^{r-2}y_a]}{[\prod_{a=1}^{r-2}y_a^{\frac{1}{2}} p][\prod_{a=1}^{r-2}y_a^{\frac{1}{2}}p^{-1}]}+\frac{[x_1]}{[x_1^{\frac{1}{2}} q_{1}][x_1^{\frac{1}{2}}q^{-1}_{1}]}+\frac{[x_r]}{[x_r^{\frac{1}{2}} q_{r}][x_r^{\frac{1}{2}}q^{-1}_{r}]},
    \end{align*}
    where in the last step we used the inductive step. Notice, in particular, that the last expression just depends on $x_1, x_r, \tilde{x}=\prod_{i=2}^{r-1}x_i, q$, and therefore  \eqref{eqn: third eqn} reduces to 
    \begin{align*}
        \frac{[\prod_{a=1}^{r-2}y_a]}{[\prod_{a=1}^{r-2}y_a^{\frac{1}{2}} p][\prod_{a=1}^{r-2}y_a^{\frac{1}{2}}p^{-1}]}+\frac{[x_1]}{[x_1^{\frac{1}{2}} q_{1}][x_1^{\frac{1}{2}}q^{-1}_{1}]}+\frac{[x_r]}{[x_r^{\frac{1}{2}} q_{r}][x_r^{\frac{1}{2}}q^{-1}_{r}]}=\frac{[\prod_{i=1}^rx_i]}{[\prod_{i=1}^rx_i^{\frac{1}{2}} q][\prod_{i=1}^rx_i^{\frac{1}{2}}q^{-1}]},
    \end{align*}
    which is a formal identity depending on the four variables $x_i, x_r ,\tilde{x}, q$, and can therefore be directly checked by comparing left-hand-side and right-hand-side\footnote{To be precise, we checked that this identity holds with the help of Mathematica.}.
\end{proof}
\begin{corollary}
    Let $\overline{r}=(r,r,r,r)$. Then for $n>0$ there is a vanishing
    \[
     \CZ_{\overline{r},n}=0.
    \]
\end{corollary}
\begin{proof}
    From $\overline{r}=(r,r,r,r)$ and $t_1t_2t_3t_4=1$ it follows that $\kappa_{\overline{r}}=1$. Therefore
    \[\frac{[\kappa_{\overline{r}}]}{[\kappa_{\overline{r}}^{\frac{1}{2}} q][\kappa_{\overline{r}}^{\frac{1}{2}}q^{-1}]}=0,
    \]
   and  by Theorem \ref{thm: explicit expression inv} it follows that the partition function of tetrahedron instantons is $ \CZ_{\overline{r}}(q)=1$.
\end{proof}
\begin{remark}
The vanishing of the instanton partition function for tuples of the form $(r,r,r,r)$ yields non-trivial relations among the coefficients of $\CZ^{(i)}(q)$. For instance, if we choose $\overline{r}=(1,1,1,1)$ and denote by $\CZ^{(i)}_{n} $ the  $n$-th coefficient  of  $\CZ^{(i)}(q)$, we obtain the relation
\[
t_1^{-\frac{1}{2}}\CZ^{(1)}_{1}+t_1^{-1}t_2^{-\frac{1}{2}}\CZ^{(2)}_{1}+t_1^{-1}t_2^{-1}t_3^{-\frac{1}{2}}\CZ^{(3)}_{1}+t_1^{-1}t_2^{-1}t_3^{-1}t_4^{-\frac{1}{2}}\CZ^{(4)}_{1}=0.
\]
This provides a new method to generate identities among classical rank 1 Donaldson-Thomas invariants  of $\BC^3$.
\end{remark}
\subsection{Cohomological limit}
Let $\overline{r}=(r_1, \dots, r_4)$ and $n\geq 0$. By Corollary \ref{cor: virtual classes}, the moduli space of tetrahedron instantons is endowed with a $\TT$-equivariant  \emph{virtual fundamental class} $[\CM_{\overline{r}, n}]^{\vir}\in A^\TT_{*}\left(\CM_{\overline{r}, n}, \BZ\left[\tfrac{1}{2}\right]\right)$. We define the $\TT$-equivariant invariants  
\begin{align*}\label{eqn: cohom invariants}
\CZ^{\coh}_{\overline{r},n}=\int_{[\CM_{\overline{r}, n}]^{\vir}}1 \in \frac{\BQ(s_1, s_2, s_3, s_4, v)}{(s_1+s_2+s_3+s_4)},
\end{align*}
where $s_1, \dots, s_4, v_{11}, \dots, v_{4r_4}$ are the generators of $\TT$-equivariant Chow cohomology $A^\TT_*(\pt)$. In other words, they can be realized as $s_i=c_1(t_i), v_{il}=c_i(w_{il})$, i.e.  the first Chern classes of the irreducible representations generating $K^0_\TT(\pt) $.

Since $\CM_{\overline{r}, n}$ is not proper, one cannot directly define  invariants via proper push-forward in equivariant cohomology. Instead,  since the fixed locus $\CM_{\overline{r}, n}^\TT $ is proper, we define invariants via $\TT$-equivariant residues on the fixed locus
\begin{align*}
   \int_{[\CM_{\overline{r}, n}]^{\vir}}1=\int_{[\CM^\TT_{\overline{r}, n}]^{\vir}}\frac{1}{\sqrt{e}(N^{\vir})},
\end{align*}
where $e(\cdot)$ is the ($\TT$-equivariant) Euler class and $ \sqrt{e}(\cdot)$ is the ($\TT$-equivariant) square-root Euler  class of Edidin-Graham \cite{EG_char_classes_quadratic_bundles} (cf. also \cite[Sec.~3.1]{OT_1}).

%as the composition
%\begin{align*}
% \int_{[\CM_{\overline{r}, n}]^{\vir}}:  A^{\TT}_*\left(\CM_{\overline{r}, n}, \BZ\left[\tfrac{1}{2}\right]\right)\to A^{\TT}_*(\CM_{\overline{r}, n})_{\mathrm{loc}}\xrightarrow{\sim} A^{\TT}_*(\CM^{\TT}_{\overline{r}, n})_{\mathrm{loc}}\to A_*^\TT(\pt)_{\mathrm{loc}}.
%\end{align*}
%Here, the first map is a suitable localisation of $ A^{\TT}_*\left(\CM_{\overline{r}, n}, \BZ\left[\tfrac{1}{2}\right]\right)$, the second map is  Edidin-Graham's abstract localisation \cite{EG_localization} and the third map is the proper pushforward on  $\CM_{\overline{r}, n}^\TT $ (extended to the localisation).

We define the \emph{cohomological tetrahedron instanton partition function} as the generating series
\begin{align*}
    \CZ^{\coh}_{\overline{r}}(q)=\sum_{n\geq 0} \CZ^{\coh}_{\overline{r},n}\cdot q^n \in \frac{\BQ(s_1, s_2, s_3, s_4, v)}{(s_1+s_2+s_3+s_4)}[\![q]\!].
\end{align*}

%By Oh-Thomas localisation theorem in equivariant cohomology  \cite{OT_1} we can compute the invariants via $\TT$-equivariant residues on the fixed locus
%\begin{align*}
%   \int_{[\CM_{\overline{r}, n}]^{\vir}}1=\int_{[\CM^\TT_{\overline{r}, n}]^{\vir}}\frac{1}{\sqrt{e}(N^{\vir})},
%\end{align*}
%where $e(\cdot)$ is the ($\TT$-equivariant) Euler class and $ \sqrt{e}(\cdot)$ is the ($\TT$-equivariant) square-root Euler  class of Edidin-Graham \cite{EG_char_classes_quadratic_bundles} (cf. also \cite[Sec.~3.1]{OT_1}).  
Therefore the tetrahedron partition function is given by
\begin{align*}
     \CZ^{\coh}_{\overline{r}}(q)=\sum_{\overline{\pi}}(-1)^{o_{\overline{\pi}}}\sqrt{e}(-T^{\vir}_{\overline{\pi}})\cdot q^{|\overline{\pi}|},
\end{align*}
where the sum runs over tuples $ \overline{\pi}=(\overline{\pi}_1, \dots, \overline{\pi}_4)$, where each $\overline{\pi}_i$ is an $r_i$-tuple of plane partitions,  $T^{\vir}_{\overline{\pi}}$ is the virtual tangent bundle at the fixed point corresponding to $\overline{\pi}$ via the correspondence in Proposition \ref{prop: fixed locus reduced} and $(-1)^{o_{\overline{\pi}}}$ is the sign of the induced virtual fundamental class.

Similarly to the $K$-theoretic setting, we can compute the square root Euler class  of $T^{\vir}_{\overline{\pi}}$ using the square root $\mathsf{v}_{\overline{\pi}}$ and obtain
\begin{align*}
     \CZ^{\coh}_{\overline{r}}(q)=\sum_{\overline{\pi}}(-1)^{\sigma_{\overline{\pi}}}e(-\mathsf{v}_{\overline{\pi}})\cdot q^{|\overline{\pi}|}.
\end{align*}
We explicitly describe how the Euler class $e(\cdot)$ acts on a virtual $\TT$-representation.  For an irreducible representation $t_1^{\mu_1}\cdots t_4^{\mu_4}w_{11}^{\mu_{11}}\cdots w_{4r_4}^{\mu_{4r_4}}$, we have
\[e(t_1^{\mu_1}\cdots t_4^{\mu_4}w_{11}^{\mu_{11}}\cdots w_{4r_4}^{\mu_{4r_4}})=\mu_1s_1+\dots +\mu_4s_4+\mu_{11}v_{11}+\dots + \mu_{4r_4}v_{4r_4}, \]
and is linear $e(t^{\mu}\pm t^{\nu})=e(t^{\mu})e(t^{\nu})^{\pm 1}$, as long as $\nu$ is not the trivial weight. To ease the notation we set $\mu=(\mu_1, \dots, \mu_4, \mu_{11}, \dots, \mu_{4r_4})$ and $s=(s_1,\dots, s_4, v_{11}, \dots, s_{4r_4})$, and write $e(t^{\mu})=\mu\cdot s$ where the product is the usual scalar product.

As explained in \cite[Sec.~7.1]{FMR_higher_rank}, one should think of $e(\cdot)$ as the \emph{linearization} of $[\cdot]$, since
\begin{align*}
    [t^{\mu}]|_{t_i=e^{bs_i}, w_{il}=e^{bv_{il}}}=e^{\frac{b\mu\cdot s}{2}}(1-e^{-b\mu\cdot s})=be(t^\mu)+O(b^2).
\end{align*}
In particular, if $V$ is a $\TT$-representation of rank 0, the limit $b\to 0$ is well-defined and yields
\begin{align}\label{eqn: lim K to cohom}
    \lim_{b\to 0}[V]|_{t_i=e^{bs_i}, w_{il}=e^{bv_{il}}}=e(V).
\end{align}
Since the vertex term  $\mathsf{v}_{\overline{\pi}} $ has rank 0, this identity tells us that the cohomological tetrahedron instanton partition function $ \CZ^{\coh}_{\overline{r}}(q) $ is a limit\footnote{Alternatively, we could show that the cohomological invariants are the limit of the $K$-theoretic invariants using virtual  Riemann-Roch \cite[Thm. 6.1]{OT_1} (cf.~\cite[Thm.~6.4]{CKM_crepant}).  } of the tetrahedron instanton partition function $ \CZ_{\overline{r}}(q)$.
\begin{corollary}\label{cor: cohom}
     Let $\overline{r}=(r_1, \dots, r_4)$  and $r=\sum_{i=1}^4r_i$. We have 
     \[\CZ^{\coh}_{\overline{r}}(q)=\mathrm{M}((-1)^rq)^{-\frac{(s_1+s_2)(s_1+s_3)(s_2+s_3)(r_1s_1+r_2s_2+r_3s_3 + r_4s_4)}{s_1s_2s_3s_4}},\]
     where $\mathrm{M}(q)$ is the MacMahon series \eqref{eqn: MacMahon}.
\end{corollary}
\begin{proof}
    By the limit \eqref{eqn: lim K to cohom}, we can compute the cohomological instanton partition function as
    \[
    \CZ^{\coh}_{\overline{r}}(q)=\lim_{b\to 0}\CZ_{\overline{r}}(q)|_{t_i=e^{bs_i}}.
    \]
    Notice in particular that the independence of $\CZ_{\overline{r}}(q)$ from the framing parameters $w_{il}$  directly implies the independence of $ \CZ^{\coh}_{\overline{r}}(q) $ from the variables $v_{il}$. Set $\overline{r}\cdot s=r_1s_1+r_2s_2+r_3s_3 + r_4s_4$. We have
    \begin{align*}
        \lim_{b\to 0}\CZ_{\overline{r}}(q)|_{t_i=e^{bs_i}}&= \lim_{b\to 0}\exp\left(\sum_{n\geq 1}\frac{1}{n}\frac{[e^{bns_1}e^{bns_2}][e^{bns_1}e^{bns_3}][e^{bns_2}e^{bns_3}]}{[e^{bns_1}][e^{bns_2}][e^{bns_3}][e^{bns_3}]}\frac{[e^{bn\overline{r}\cdot s}]}{[e^{\frac{bn\overline{r}\cdot s}{2}} q][e^{\frac{bn\overline{r}\cdot s}{2}}q^{-1}]}\right)\\
        &= \exp \left(\frac{(s_1+s_2)(s_1+s_3)(s_2+s_3)(r_1s_1+r_2s_2+r_3s_3 + r_4s_4)}{s_1s_2s_3s_4}\sum_{n\geq 1}\frac{1}{n}\frac{-q^n}{(1-q^n)^2} \right).
    \end{align*}
    Recall the plethystic expression for the MacMahon series
\[
\mathrm{M}(q)=\Exp\left(\frac{q}{(1-q)^2} \right),
\]
by which we conclude that 
    \[
      \lim_{b\to 0}\CZ_{\overline{r}}(q)|_{t_i=e^{bs_i}}=\mathrm{M}((-1)^rq)^{-\frac{(s_1+s_2)(s_1+s_3)(s_2+s_3)(r_1s_1+r_2s_2+r_3s_3 + r_4s_4)}{s_1s_2s_3s_4}}.
    \]
\end{proof}
\subsection{Witten genus}\label{sec: witten}
For a scheme endowed with a perfect obstruction theory in the sense of Behrend-Fantechi \cite{BF_normal_cone}, the authors \cite{FMR_higher_rank} introduced the \emph{virtual chiral elliptic genus},
inspired by the work of Benini-Bonelli-Poggi-Tanzini in String Theory \cite{BBPT_elliptic_DT}. In the context of \cite{FMR_higher_rank}, the \emph{chirality} refers to the novel feature of  extracting square roots of the virtual canonical bundle, as opposed to the more classical \emph{virtual elliptic genus} of Fantechi-Göttsche \cite{FG_riemann_roch}. From the complex-analytic point of view, the virtual chiral elliptic genus is the virtual analogue of the \emph{Witten genus}, which refines the $\widehat{A}$-genus (cf.~\cite{Witten_elliptic}), while the (virtual) elliptic genus refines the (virtual) $\chi_y$-genus.

More recently, for a moduli space of sheaves on a Calabi-Yau fourfold, Bojko \cite[Def. 5.8]{Bojko_wall-crossing} defined the  \emph{DT4 Witten genus}, which can be seen as the generalization of the virtual chiral elliptic genus to virtual structures à la Oh-Thomas. We adapt his notion to our setting to define the  \emph{virtual Witten genus} of the moduli space of tetrahedron instantons.

For a vector bundle $F$ on a scheme $X$, define the \emph{total symmetric power} and \emph{total exterior power}
\begin{align*}
    \Sym^\bullet_{p}F&=\bigoplus_{n\geq 0}p^n\Sym^n F\in K^0(X)[\![p]\!],\\
     \Lambda^\bullet_{p}F&=\bigoplus_{n\geq 0}p^n\Lambda^n F\in K^0(X)[p],
\end{align*}
and extend them by linearity to $K^0(X)$.  They satisfy
\begin{align*}
      \Sym^\bullet_{p}F=\frac{1}{\Lambda^\bullet_{-p}F}.
\end{align*}
\begin{definition}\label{def:elliptic_invariant}
    Let $\overline{r}=(r_1, \dots, r_4)$. The \emph{virtual Witten genus} of $\CM_{\overline{r}, n} $ is defined as
    \[
    \CZ^{\mathrm{ell}}_{\overline{r}, n}=\chi\left(\CM_{\overline{r}, n}, \widehat{\oO}_{\CM_{\overline{r}, n}}^{\vir}\otimes \bigotimes_{n\geq 1} \Sym^\bullet_{p^n}\left(T^{\vir}_{\CM_{\overline{r}, n}}\right) \right)\in 
     \frac{\BQ(t_1^{ \frac{1}{2}}, t_2^{ \frac{1}{2}}, t_3^{ \frac{1}{2}}, t_4^{ \frac{1}{2}}, w^{ \frac{1}{2}})}{(t_1t_2t_3t_4-1)}[\![p]\!].\]
\end{definition}
We refer to $p$ as the \emph{elliptic parameter} and define the \emph{elliptic tetrahedron instanton partition function} as the generating series
\begin{align*}
    \CZ^{\mathrm{ell}}_{\overline{r}}(q)=\sum_{n\geq 0} \CZ^{\mathrm{ell}}_{\overline{r},n}\cdot q^n \in \frac{\BQ(t_1^{ \frac{1}{2}}, t_2^{ \frac{1}{2}}, t_3^{ \frac{1}{2}}, t_4^{ \frac{1}{2}}, w^{ \frac{1}{2}})}{(t_1t_2t_3t_4-1)}[\![p,q]\!].
\end{align*}
As for the case of $K$-theoretic invariants, the moduli space $ \CM_{\overline{r}, n}$ is not proper, but its torus fixed locus $\CM_{\overline{r}, n}^\TT$ is, so we define the invariants via equivariant residues \eqref{def: inv by loc def}.

 The \emph{Jacobi theta function} $\theta(p;y)$ and the \emph{Dedekind eta function} $\eta(p)$ are defined as
 \begin{align*}
    \theta(p;y)&=-ip^{1/8}(y^{1/2}-y^{-1/2})\prod_{n=1}^\infty(1-p^n)(1-yp^n)(1-y^{-1}p^n), \\
    \eta(p)&=p^{\frac{1}{24}}\prod_{n=1}^\infty(1-p^n),
 \end{align*}
 and denote $\theta(\tau|z):=\theta(e^{2\pi i\tau};e^{2\pi iz})$. If we set $p=e^{2\pi i \tau}$, with $\tau\in \mathbb{H}=\set{\tau\in \BC | \mathrm{Im}(\tau)>0}$, then  $\theta$ enjoys the  modular behaviour %\cite{Mumford_Tata_I}
\[
\theta(\tau|z+a+b\tau)=(-1)^{a+b}e^{-2\pi ibz}e^{-i\pi b^2\tau}\theta(\tau|z),\quad a,b\in\BZ.
\]
See \cite[Sec.~8.1]{FMR_higher_rank} and \cite[Sec.~6]{FG_riemann_roch} for a related discussion on the modularity of these functions.   For an irreducible $\TT$-representation $t^\mu$,  define (cf.~\cite[Sec.~8.1]{FMR_higher_rank})
     \[\theta[t^{\mu}]=  (i\cdot \eta(p))^{-1}\theta(p;t^\mu)\in  K^0_\TT(\pt)\llbracket p  \rrbracket[p^{\pm \frac{1}{12}}]\] 
and extend it by linearity to any  $V\in K_\TT^0(\pt)$. Notice that $\theta[t^{\mu}] $ satisfies $\theta[t^{-\mu}]=-\theta[t^{\mu}]$ and if $V$ is a virtual representation such that $\rk V=0$, the elliptic measure refines both $[\cdot]$ and $e(\cdot)$
\[
\theta[V]\xrightarrow{p\to 0} [V]\xrightarrow{b\to 0} e(V).
\]
Let $\mathsf{v}_{\overline{\pi}}$ our preferred  square root of the virtual tangent space  $T^{\vir}_{\overline{\pi}}$ at the $\TT$-fixed point corresponding to $\overline{\pi}$ and recall that the sign appearing in the localisation satisfies $ (-1)^{\sigma_{\overline{\pi}}}=1$. By Oh-Thomas localisation \cite{OT_1}, we have
\begin{align*}
     \CZ^{\mathrm{ell}}_{\overline{r}}(q)&=\sum_{\overline{\pi}}\frac{ \prod_{n\geq 1} \Sym^\bullet_{p^n}\left(T^{\vir}_{\overline{\pi}}\right)}{\sqrt{\mathfrak{e}}(T^{\vir}_{\overline{\pi}})}
     \cdot q^{|\overline{\pi}|}\\
     &=\sum_{\overline{\pi}}\frac{ \prod_{n\geq 1} \Sym^\bullet_{p^n}\left(\mathsf{v}_{\overline{\pi}}+\overline{\mathsf{v}_{\overline{\pi}}} \right)}{\mathfrak{e}(\mathsf{v}_{\overline{\pi}})\cdot ({\det}\mathsf{v}_{\overline{\pi}})^{\frac{1}{2}}}
     \cdot q^{|\overline{\pi}|}\\
     &=\sum_{\overline{\pi}}\theta[-\mathsf{v}_{\overline{\pi}}] \cdot q^{|\overline{\pi}|},
\end{align*}
where the last line follows from the explicit computation in \cite[pag. 34]{FMR_higher_rank}.

\begin{remark}\label{rmk: failure witten}
    The elliptic tetrahedron instanton partition function $ \CZ^{\mathrm{ell}}_{\overline{r}}(q)$ does in general depend on the framing parameters $w_{il}$, differently from its $K$-theoretical and cohomological analogues, where a rigidity principle applies. See \cite[Ex. 8.7]{FMR_higher_rank} for the case of $r=(0,0,0,3)$ and $n=1$.
\end{remark}
\begin{remark}\label{rem: rigidity witten genus}
   We discussed in  \cite[Sec. 8.3]{FMR_higher_rank} --- in the context of higher rank Donaldson-Thomas theory --- how specialising the equivariant parameters to roots of unity yields, once more, a rigidity phenomenon for which the dependence on the elliptic parameter disappears and the partition function is computed explicitly \cite[Thm. 8.11]{FMR_higher_rank}. By computational evidence, we noticed a similar phenomenon and expect it to  hold for the elliptic tetrahedron instanton partition function, where one suitably specialises the equivariant parameters $t_1, \dots, t_4$ to $r_i$-roots of unity. See \cite[Rem. 8.12]{FMR_higher_rank} for its physical role from the point of view of supersymmetry.
\end{remark}

\subsection{Pomoni-Yan-Zhang} \label{sec: PYZ}
As we pointed out in \S\,\ref{sec:intro tetrahedra}, our work aims to give a geometrisation of the moduli spaces originally considered in \cite{PYZ_tetrahedron}, to get a better understanding of their structure and the relation they bear to enumerative theories already extensively studied. In loc.~cit., the authors gave a description of a moduli space of BPS solutions to the D1-brane low-energy effective theory of a system of D1-D7-branes on the type IIB supersymmetric background $X=\BR^{1,1}\times\BR^8$. In particular, a stack of $n$ coincident D1-branes probes a configuration $r=r_1+r_2+r_3+r_4$ of intersecting D7-branes, where each stack of $r_i$ coincident D7-branes wraps an irreducible component of the boundary toric divisor $\Delta\subset\BC^4\cong\BR^8$. The resulting low-energy effective theory on the D1-branes is a so-called $2d$ $\mathcal N=(0,2)$ quiver Gauged Linear Sigma Model, whose partition function can be explicitly computed by means of supersymmetric localisation as in \cite{Benini-elliptic-genus-rank-1,Benini-elliptic-genus-rank-r}. The advantage of this procedure is that it does not require an in-depth understanding of the moduli space at hand, as the partition function of the theory is only dependent on the field content (which is entirely encoded in the datum of the quiver and possibly the relations), and it is expressed in terms of so-called Jeffrey-Kirwan residues \cite{JK}, via a well-defined procedure. In this section, we will briefly describe the constructions in loc.~cit., and explain how the connection to geometry is drawn.

In \cite[Sec. 2]{PYZ_tetrahedron}, the moduli space of tetrahedron instantons is introduced via the study of vacua of the low-energy theory of the stack of D1-branes in the background of the D7-branes. In particular, the effective theory has a family of classical vacua, which are characterised by the vanishing of the scalar potential and are parametrised by the symplectic quotient by the unitary group $\U_n$
\[
\mathfrak M_{\overline r,n}(\zeta)\defeq\mu_{\BR}^{-1}(\zeta)\cap\mu_{\BC}^{-1}(0)\cap\sigma^{-1}(0)/\U_n,\quad \zeta\in\mathfrak u_n^\ast,
\]
where the moment map $\mu_\BR:\mathsf R_{\overline r,n}\to\mathfrak u_n$ is defined by
\[
\mu_{\BR}(B_1,\dots,B_4,I_1,\dots,I_4)=\frac{\sqrt{-1}}{2}\left(\sum_{a=1}^4[B_a,B_a^\dagger]+\sum_{i=1}^4I_iI_i^\dagger\right),
\]
while the maps $\mu_{\BC}=(\mu_{\BC})_{a,b}$, for $a,b=1,\dots,4$ and $\sigma=(\sigma)_i$, for $i=1,\dots,4$, act on representations in $\mathsf R_{\overline r,n}$ as
\begin{align*}
    \mu_{\BC}(B_1,\dots,B_4,I_1,\dots,I_4)&=\left([B_a,B_b]\right)_{a,b=1,\dots,4},\\
    \sigma(B_1,\dots,B_4,I_1,\dots,I_4)&=\left(B_iI_i\right)_{i=1,\dots,4}.
\end{align*}
By Theorem \ref{thm: isotropic construction} and \cite[Cor.~3.22]{Nak_lectures_Hilb_schemes}, there is a bijection between the symplectic quotient $\mathfrak M_{\overline r,n}(\sqrt{-1}d\chi)$ and   (the closed points of) the Quot scheme $\Quot_\Delta(\mathcal E,n)$, the latter being realised as a GIT quotient with respect to the $\GL_n$-character $\chi$. Moreover, in \cite[\S\,4.3.2, \S\,4.3.3]{PYZ_tetrahedron}, a stratification of the one-instanton moduli space $\CM_{\overline r,1}$ is obtained from the gauge-theoretic description derived from physics. This stratification agrees with our Example~\ref{example: blow up}, where it is also interpreted geometrically in terms of blow-ups.

The $n$-instantons contribution to the partition function of the two-dimensional low-energy theory on the worldvolume of the D1-branes is then captured by a refinement of the Witten index. This is known in physics as \textit{elliptic genus}, and is defined as in \cite[Eq.~(163)]{PYZ_tetrahedron}. Notably, this refined index can be computed exactly for two-dimensional Gauged Linear Sigma Models via supersymmetric localisation (cf.~\cite[Eq.~(3.7)]{Benini-elliptic-genus-rank-1} for rank-1 gauge groups and \cite[Eq.~(2.27)]{Benini-elliptic-genus-rank-r} for general rank), and its evaluation involves the computation of certain equivariant residues, the Jeffrey-Kirwan residues \cite{JK}, which are only dependent on the (bosonic) field content of the theory through their 1-loop contribution to the path integral formulation of the elliptic genus. When the moduli space engineered by the BPS vacua of the $2d$ $\mathcal N=(0,2)$ Gauged Linear Sigma Model is equipped with a perfect obstruction theory in the sense of Li-Tian \cite{LT_virtual_cycle} and Behrend-Fantechi \cite{BF_normal_cone}, the physical elliptic genus has a geometric realisation as the \textit{(virtual) chiral elliptic genus}, first computed in the context of DT theory in \cite{BBPT_elliptic_DT}, and later defined in \cite{FMR_higher_rank}. In the case of tetrahedron instantons, however, a perfect obstruction theory is not available, but virtual structures \`a la Oh-Thomas are induced by the zero-locus construction of Thm.~\ref{thm: isotropic construction}. Then the elliptic genus computed in \cite[Eqs.~(211)-(214)]{PYZ_tetrahedron} is recovered exactly from our computation of the virtual Witten genus in \S\,\ref{sec: witten}. In this sense, one might view supersymmetric localisation as nothing but an infinite-dimensional analogue of equivariant localisation. However, despite the undoubted advantage that an in-depth knowledge of the structure of the moduli space of BPS vacua is not required for the calculation of the elliptic invariants in Def.~\ref{def:elliptic_invariant}, carried out as in \cite[\S\,5]{PYZ_tetrahedron}, it nevertheless overshadows some interesting features which might be worth pointing out. Indeed, the fact that the virtual structure on $\CM_{\overline r,n}$ is induced, via the seminal work of Oh-Thomas \cite{OT_1}, by the construction in Thm.~\ref{thm: isotropic construction}, means that the invariants computed in K-theory by localisation will depend on a choice of orientation via a sign, cf.~\S\,\ref{sec:invariants}. This sign, however, is not manifestly present in \cite[\S\S\,5-6]{PYZ_tetrahedron}. Moreover, the choice of a square root of the virtual tangent bundle as in \S\,\ref{sec:invariants} is not unique, but the couple square root/sign is somewhat canonical (cf.~Thm.~\ref{thm: KR sign} and \cite{Mon_canonical_vertex}). Therefore, while the character $\chi_{\overline\pi}$ in \cite[Eq.~(147)]{PYZ_tetrahedron} is a good choice for a square root of the virtual tangent space at the points in the fixed locus $\CM_{\overline r,n}^{\TT}$, it is not the only possible one (a different choice being given by our vertex term $\mathsf v_{\overline\pi}$ in Eq.~\ref{eqn: vertex term}), and in the computation of the invariants \cite[Eq.~(153)]{PYZ_tetrahedron}, it should in principle be accompanied by the additional choice of an appropriate sign. Similarly, the observation in loc.~cit. that the partition function thereby computed is invariant only up to a sign under permutations of the $\kappa_i$, is a shadow of the same phenomenon.

\appendix
\section{The sign rule}\label{sec: app}
We provide in this appendix a proof of Theorem \ref{thm: correct sign body text}, which computes the correct sign for each localised contribution of the tetrahedron instantons partition function. Our proof relies on two results of Kool-Rennemo, whose proofs will appear in their forthcoming paper \cite{KR_magnificient}.\\

  Let a scheme $Z$
\[
\begin{tikzcd}
& \CE\arrow[d]\\
Z:=Z(s)\arrow[r, hook, "\iota"] &\CA\arrow[u, bend right, swap, "s"]
\end{tikzcd}
\]
be the zero locus of an isotropic section $s\in \Gamma(\CA,\CE)$, where $\CE$ is a $SO(2r, \BC)$-bundle over a smooth quasi-projective variety $\CA$, so that $Z$ is endowed with the obstruction theory \eqref{eqn: obs th}.
Assume that the orientation of $\CE$ is induced by the choice of a maximal isotropic subbundle  $\Lambda\subset \CE$, which induces a projection
\[
p_\Lambda:\CE^*\to \Lambda.
\]
 Let  $\CA$ be acted by an algebraic torus $\BT$, such that $(\CE,q)$ is $\BT$-equivariant and the $\BT$-action restricts to $Z$, which implies that the obstruction theory $\BE\to \BL_Z$ is $\BT$-equivariant and $Z$ is endowed with a $\BT$-equivariant virtual structure sheaf $\widehat{\oO}^{\vir}_Z$. Finally, we assume that the fixed locus $Z^\BT$ is reduced, zero-dimensional, and such that the fibre of the obstruction theory  $\BE|_p$ is $\BT$-movable, for every fixed point $p\in Z^\BT$. 
\begin{theorem}[{Kool-Rennemo \cite{KR_magnificient}}]\label{thm: KR sign}
    We have that 
    \begin{align*}
        \chi\left(Z, \widehat{\oO}^{\vir}_Z\right)=\sum_{p\in Z^\BT}(-1)^{\psi_p}[-T_{\CA}|_{p}+\Lambda|_p],
    \end{align*}
    where the sign is defined as
    \[
   \psi_p= \dim\left(\coker(T_{\CA}|_{p}\xrightarrow{(p_\Lambda\circ ds)|_p} \Lambda|_p)^{\fix} \right) \mod 2.
    \]
\end{theorem}
Notice that the sign $(-1)^{\psi_p}$ in Theorem  \ref{thm: KR sign} depends on the chosen maximal isotropic subbundle $\Lambda$, which affects the orientation on $Z $ and therefore the one on the fixed locus $Z^\BT$.\\

The assumptions of Theorem \ref{thm: KR sign}  are clearly satisfied by the  $\TT$-action  on the moduli space of tetrahedron instantons $\CM_{\overline{r},n}$, whose orientation is induced by the maximal isotropic subbundle $ \Lambda \subset \CL$ as constructed in Section \ref{sec: zero locus}. Let $\overline{\pi}$ be a tuple of plane partitions  corresponding to a fixed point of $\CM_{\overline{r},n}^\TT$. Then
  \[
\CL|_{\overline{\pi}}=\Lambda^2\BC^4\otimes \Hom(Q_{\overline{\pi}}, Q_{\overline{\pi}})\oplus\bigoplus_{i=1}^4\Hom(K_i\cdot t_i, Q_{\overline{\pi}})\oplus \bigoplus_{i=1}^4\Hom(K_i\cdot t_i, Q_{\overline{\pi}})^*
   \]
is a vector space, endowed with a non-degenerate quadratic pairing, with maximal isotropic subbundle 
\begin{align*}
    \Lambda|_{\overline{\pi}}&=\tilde{\Lambda}|_{\overline{\pi}}\oplus \bigoplus_{i=1}^4\Hom(K_i\cdot t_i, Q_{\overline{\pi}})^*      \\
    \tilde{\Lambda}|_{\overline{\pi}}&=\left(\langle e_4\rangle\wedge \BC^3\right)\otimes \Hom(Q_{\overline{\pi}}, Q_{\overline{\pi}}),
\end{align*}
   and isotropic section
   \[
   s|_{\overline{\pi}}=(\tilde{s}|_{\overline{\pi}},0):\Lambda|_{\overline{\pi}}\to \CL|_{\overline{\pi}}.
   \]
   By Theorem \ref{thm: KR sign}, we have that 
      \begin{align*}
        \chi\left(\CM_{\overline{r},n}, \widehat{\oO}^{\vir}_{\CM_{\overline{r},n}}\right)=\sum_{\overline{\pi}}(-1)^{\psi_{\overline{\pi}}}[-T_{\CM_{\overline{r},n}^{\nc}}|_{\overline{\pi}}+\Lambda|_{\overline{\pi}}],
    \end{align*}
    where
    \begin{align*}
\psi_{\overline{\pi}}=\dim\left(\coker(T_{\CM_{\overline{r},n}^{\nc}}|_{p}\xrightarrow{(p_\Lambda\circ ds)|_{\overline{\pi}}} \Lambda|_{\overline{\pi}})^{\fix} \right) \mod 2.
    \end{align*}
    By Proposition \ref{prop: t movable} the vertex term $\mathsf{v}_{\overline{\pi}}$ is $\TT$-movable, which implies that the virtual tangent space $T^{\vir}_{\overline{\pi}}$ is $\TT$-movable as well. In particular, any other square root of $ T^{\vir}_{\overline{\pi}}$ has to be $\TT$-movable. Therefore,  $ T_{\CM_{\overline{r},n}^{\nc}}|_{\overline{\pi}}-\Lambda|_{\overline{\pi}}$ must be  a $\TT$-movable virtual representation.
\begin{lemma}\label{lemma: appendix}
    We have 
    \begin{align*}
        \psi_{\overline{\pi}}=\dim\left(\coker(T_{\CM_{\overline{r},n}^{\nc}}|_{p}\xrightarrow{(p_{\tilde{\Lambda}}\circ d\tilde{s})|_{\overline{\pi}}} \tilde{\Lambda}|_{\overline{\pi}} )^{\fix} \right) \mod 2.
    \end{align*}
\end{lemma}
\begin{proof}
    The differential of the isotropic section satisfies $ds=(d\tilde{s}, 0)$, therefore 
    \begin{align*}
      \coker(T_{\CM_{\overline{r},n}^{\nc}}|_{p}\xrightarrow{(p_\Lambda\circ ds)|_{\overline{\pi}}} \Lambda|_{\overline{\pi}})=\coker(T_{\CM_{\overline{r},n}^{\nc}}|_{p}\xrightarrow{(p_{\tilde{\Lambda}}\circ d\tilde{s})|_{\overline{\pi}}} \tilde{\Lambda}|_{\overline{\pi}} )\oplus \bigoplus_{i=1}^4\Hom(K_i\cdot t_i, Q_{\overline{\pi}})^*.
    \end{align*}
   The proof follows by the vanishing of the fixed part
    \begin{align*}
        \left(\Hom(K_i\cdot t_i, Q_{\overline{\pi}})\right)^{\fix}&=\left(\overline{K_i}\kappa_iQ_{\overline{\pi}}\right)^{\fix}\\
        &= \left(\sum_{j=1}^4\sum_{\substack{1\leq l\leq r_i\\ 1\leq k\leq r_j}}w_{il}^{-1}w_{jk}\kappa_iZ_{\pi_{jk}}\right)^{\fix}\\
       &= \left(  \sum_{l=1}^{r_i}\kappa_i Z_{\pi_{il}} \right)^{\fix}\\
       &=0,
    \end{align*}
    since $Z_{\pi_{il}}$ is a polynomial in the variables $t_{i_1},t_{i_2},t_{i_3}$, where $\{i_1, i_2, i_3\}$ are the three indices in $\{1,2,3,4\}$ different from $i$.
\end{proof}
The vanishing in the  proof of Lemma \ref{lemma: appendix} readily implies that $T_{\CM_{\overline{r},n}^{\nc}}|_{\overline{\pi}}-\tilde{\Lambda}|_{\overline{\pi}}$ is a $\TT$-movable virtual representation.

For every tuple of plane partitions $\overline{\pi}$, define the Laurent polynomial 
\begin{align*}
   \tilde{\mathsf{v}}_{\overline{\pi}}=\overline{K}Q_{\overline{\pi}}-\overline{P_{123}}Q_{\overline{\pi}}\overline{Q_{\overline{\pi}}}.
\end{align*}
The vertex term $ \tilde{\mathsf{v}}_{\overline{\pi}} $ is the higher rank version of the vertex term of Cao-Kool-Monavari \cite[Eqn. (18)]{CKM_K_theoretic}, used in loc. cit. to compute the Donaldson-Thomas invariants of $\BC^4$, and is easily seen to be $\TT$-movable by \cite[Lemma 2.4]{Mon_canonical_vertex}. Kool-Rennemo \cite{KR_magnificient} exploited this vertex term to prove the closed formula for Donaldson-Thomas invariants of $\BC^4$ conjectured by Nekrasov-Piazzalunga \cite{NP_colors}, by explicitly computing the corresponding sign for every localised contribution. Notice that the immersion 
 \begin{align*}
     \CM_{\overline{r}, n}\hookrightarrow \Quot_{\BC^4}(\oO_{\BC^4}^r,n)
     %\hookrightarrow \mathcal{M}^{\mathrm{nc}}_{\overline{r}, n},
 \end{align*}
induces a closed immersion of their $\TT$-fixed locus
\begin{align*}
     \CM_{\overline{r}, n}^\TT\hookrightarrow \Quot_{\BC^4}(\oO_{\BC^4}^r,n)^\TT.
     %\hookrightarrow \mathcal{M}^{\mathrm{nc}}_{\overline{r}, n},
 \end{align*}
The right-hand-side is reduced, zero-dimensional and in bijection with tuples of \emph{solid partitions} of size $n$, which can be seen as \emph{box arrangements} in $\BZ^4_{\geq 0}$ (cf.~\cite[Sec.~2.1]{Mon_canonical_vertex} for the rank 1 case), generalizing plane partitions to one dimension higher. Under this identification, a tuple of plane partitions $\overline{\pi}=(\overline{\pi}_i)_{i=1, \dots, 4}$ can be seen as a tuple of solid partitions, such that each plane partition $\pi_{il}$ -- for $i=1, \dots, 4$ and $l=1, \dots, r_i$ -- is seen as a solid partition with the $i$-coordinate equal to zero.

A key step of the computation of Kool-Rennemo \cite{KR_magnificient} yields the following combinatorial result.
\begin{prop}[{Kool-Rennemo \cite{KR_magnificient}}]\label{prop_ KR explciit sign}
    Let $\overline{\pi}=(\overline{\pi}_1, \dots \overline{\pi}_4)$, where $\overline{\pi}_i=(\pi_{i1}, \dots, \pi_{ir_i})$ are tuples of plane partitions. Then 
    \[
   (-1)^{\dim\left(\coker(T_{\CM_{\overline{r},n}^{\nc}}|_{p}\xrightarrow{(p_{\tilde{\Lambda}}\circ d\tilde{s})|_{\overline{\pi}}} \tilde{\Lambda}|_{\overline{\pi}} )^{\fix} \right) }[-T_{\CM_{\overline{r},n}^{\nc}}|_{\overline{\pi}}+\tilde{\Lambda}|_{\overline{\pi}}]=(-1)^{\rho_{\overline{\pi}}}[-  \tilde{\mathsf{v}}_{\overline{\pi}}],
    \]
    where the sign is defined as
    \begin{align*}
\rho_{\overline{\pi}}=\sum_{i=1}^{4}\sum_{l=1}^{r_i}\left|\{(a,a,a,d)\in \pi_{il}: a<d\}\right| \mod 2.
    \end{align*}
\end{prop}
The sign $\rho_{\overline{\pi}}$ was firstly conjecture to be of this form by Nekrasov-Piazzalunga \cite{NP_colors} (cf.~\cite{CKM_K_theoretic, Mon_canonical_vertex} for the rank 1 theory). We remark once more that we are seeing here the plane partition $\pi_{il}$ as a solid partition, via the embedding $\BZ^3_{\geq 0} \subset \BZ^4_{\geq 0}$ given by imposing the $i$-th coordinate to be zero.
\begin{corollary}\label{cor: final sign}
    We have
    \begin{align*}
           (-1)^{\psi_{\overline{\pi}}}[-T_{\CM_{\overline{r},n}^{\nc}}|_{\overline{\pi}}+\Lambda|_{\overline{\pi}}]=[-\mathsf{v}_{\overline{\pi}}].
    \end{align*}
    In particular, the sign $(-1)^{\sigma_{\overline{\pi}}}=1 $ is constant.
\end{corollary}
\begin{proof}
Proposition \ref{prop_ KR explciit sign} yields
\begin{align*}
      (-1)^{\psi_{\overline{\pi}}}[-T_{\CM_{\overline{r},n}^{\nc}}|_{\overline{\pi}}+\Lambda|_{\overline{\pi}}]=(-1)^{\rho_{\overline{\pi}}}[-  \tilde{\mathsf{v}}_{\overline{\pi}}+ \sum_{i=1}^4K_i t_i \overline{Q_{\overline{\pi}}}].
\end{align*}
To conclude the argument, we need to show that 
\begin{align}\label{eqn: proof sign}
    (-1)^{\rho_{\overline{\pi}}}[-  \tilde{\mathsf{v}}_{\overline{\pi}}+ \sum_{i=1}^4K_i t_i \overline{Q_{\overline{\pi}}}]=[ -\mathsf{v}_{\overline{\pi}}].
\end{align}
By the decomposition of the vertex term \eqref{eqn: decomposition vertex} and Proposition \ref{prop_ KR explciit sign},   Equation \eqref{eqn: proof sign} holds if, equivalently, the families of equations \eqref{eqn: equiv 1}, \eqref{eqn: equiv 2} hold. The first one asks that for every $(i,l)$
\begin{align}\label{eqn: equiv 1}
     (-1)^{\rho_{{\pi_{il}}}}[Z_{il}-
 \overline{P_{123}}Z_{il}\overline{Z_{il}}]&=[Z_{il}-
 \overline{P_{i_1i_2i_3}}Z_{il}\overline{Z_{il}}],
 \end{align}
 where the sign is defined as
 \begin{align*}
\rho_{\pi_{il}}=\left|\{(a,a,a,d)\in \pi_{il}: a<d\}\right| \mod 2.
    \end{align*}
 The second asks that  for every $(i,l)<(j,k)$
 \begin{multline}\label{eqn: equiv 2}
     [  (w_{il}^{-1}w_{jk}\left(Z_{\pi_{jk}}-\kappa^{-1}_j\overline{ Z_{\pi_{il}}}-\overline{P_{123}}Z_{\pi_{jk}}\overline{ Z_{\pi_{il}}} \right)+  w_{il}w_{jk}^{-1}\left(Z_{\pi_{il}}-\kappa^{-1}_i\overline{ Z_{\pi_{jk}}}-\overline{P_{123}}Z_{\pi_{il}}\overline{ Z_{\pi_{jk}}}\right)]\\
     = [  w_{il}^{-1}w_{jk}\left(Z_{\pi_{jk}}-\kappa^{-1}_j\overline{ Z_{\pi_{il}}}-\overline{P_{j_1j_2j_3}}Z_{\pi_{jk}}\overline{ Z_{\pi_{il}}} \right)+  w_{il}w_{jk}^{-1}\left(Z_{\pi_{il}}-\kappa^{-1}_i\overline{ Z_{\pi_{jk}}}-\overline{P_{j_1j_2j_3}}Z_{\pi_{il}}\overline{ Z_{\pi_{jk}}}\right)].
 \end{multline}
 Equation \eqref{eqn: equiv 2} is equivalent to 
\begin{multline}\label{eqn: appendix rewritten}
     [  w_{il}^{-1}w_{jk}\overline{P_{123}}Z_{\pi_{jk}}\overline{ Z_{\pi_{il}}} +  w_{il}w_{jk}^{-1}\overline{P_{123}}Z_{\pi_{il}}\overline{ Z_{\pi_{jk}}}]\\
     = [  w_{il}^{-1}w_{jk}\overline{P_{j_1j_2j_3}}Z_{\pi_{jk}}\overline{ Z_{\pi_{il}}} +  w_{il}w_{jk}^{-1}\overline{P_{j_1j_2j_3}}Z_{\pi_{il}}\overline{ Z_{\pi_{jk}}}].
 \end{multline}
To prove \eqref{eqn: appendix rewritten}, we just need to show that
\begin{multline*}
w_{il}^{-1}w_{jk}\overline{P_{123}}Z_{\pi_{jk}}\overline{ Z_{\pi_{il}}} +  w_{il}w_{jk}^{-1}\overline{P_{123}}Z_{\pi_{il}}\overline{ Z_{\pi_{jk}}}\\
    -\left( w_{il}^{-1}w_{jk}\overline{P_{j_1j_2j_3}}Z_{\pi_{jk}}\overline{ Z_{\pi_{il}}}+  w_{il}w_{jk}^{-1}\overline{P_{j_1j_2j_3}}Z_{\pi_{il}}\overline{ Z_{\pi_{jk}}}\right)=U-\overline{U},
\end{multline*}
for a certain Laurent polynomial $U$ such that $(-1)^{\rk U^{\mov}}=1$. In fact, in this case we would have that 
\begin{multline*}
     [  w_{il}^{-1}w_{jk}\overline{P_{123}}Z_{\pi_{jk}}\overline{ Z_{\pi_{il}}} +  w_{il}w_{jk}^{-1}\overline{P_{123}}Z_{\pi_{il}}\overline{ Z_{\pi_{jk}}}\\
      -  \left(w_{il}^{-1}w_{jk}\left(-\overline{P_{j_1j_2j_3}}Z_{\pi_{jk}}\overline{ Z_{\pi_{il}}} \right)+  w_{il}w_{jk}^{-1}\left(-\overline{P_{j_1j_2j_3}}Z_{\pi_{il}}\overline{ Z_{\pi_{jk}}}\right)\right)]=\frac{[U^{\mov}]}{[\overline{U^{\mov}}]}\\
      =(-1)^{\rk U^{\mov}}=1,
 \end{multline*}
which would show Equation \eqref{eqn: equiv 2}. Using the Calabi-Yau condition $t_1t_2t_3t_4=1$ it is easy to check that the following identity holds
\[
 \overline{P_{123}}-\overline{P_{j_1j_2j_3}}=-\overline{\left( \overline{P_{123}}-\overline{P_{j_1j_2j_3}}\right)}.
\]
Using this relation, we may choose $U$ of the form
\begin{align*}
    U=w_{il}^{-1}w_{jk}Z_{\pi_{jk}}\overline{ Z_{\pi_{il}}}( \overline{P_{123}}-\overline{P_{j_1j_2j_3}}),
\end{align*}
which satisfies
\begin{align*}
    \rk U^{\mov}&= \rk U^{\fix} \mod 2\\
    &=0,
\end{align*}
which implies that $ (-1)^{\rk U^{\mov}}=1$ as requested.

In order to prove \eqref{eqn: equiv 1} -- for every $(i,l)$ -- we exploit \cite[Thm. 2.8]{Mon_canonical_vertex}\footnote{See \cite[Sec.~3.1]{Mon_canonical_vertex} for the result  \cite[Thm. 2.8]{Mon_canonical_vertex} spelled out for the  $K$-theoretic refined invariants and the operator $[\cdot]$.}, which proves that
\begin{align*}
     (-1)^{\rho_{{\pi_{il}}}}[Z_{il}-
 \overline{P_{123}}Z_{il}\overline{Z_{il}}]&=(-1)^{\tilde{\rho}_{{\pi_{il}}}}[Z_{il}-
 \overline{P_{i_1i_2i_3}}Z_{il}\overline{Z_{il}}],
 \end{align*}
 where the sign is defined as
  \begin{align*}
\tilde{\rho}_{\pi_{il}}=\left|\{(a_1,a_2,a_3,a_4)\in \pi_{il}: a_{i_1}=a_{i_2}=a_{i_3}<a_i\}\right| \mod 2.
    \end{align*}
We conclude by noticing that the plane partition  $\pi_{il}$ --- when seen as a solid partition --- satisfies $a_{i}=0$ for all $(a_1, a_2, a_3, a_4)\in \pi_{il}$, which implies that $\tilde{\rho}_{\pi_{il}}=0$.
\end{proof}
\begin{remark}
 In Corollary \ref{cor: final sign} we used the computation of the sign of Kool-Rennemo (Proposition \ref{prop_ KR explciit sign}), combined with the canonicity of the vertex term \cite[Thm. 2.8]{Mon_canonical_vertex}, to prove that the sign is constantly 1. However, we could in principle directly compute the relevant sign from Theorem \ref{thm: KR sign}, and then show that the resulting sign for our choice of square root $\mathsf{v}_{\overline{\pi}}$ is constantly 1 by adapting the  proof of \cite[Thm. 2.8]{Mon_canonical_vertex} to our context.
\end{remark}

%\subsection{Kool-Rennemo}
%\subsection{specialisation}
 %\printbibliography
\bibliographystyle{amsplain-nodash}
\bibliography{The_Bible}
\end{document}